\documentclass{article}
\pdfoutput=1

\usepackage[utf8]{inputenc}
\usepackage[a4paper,left=3cm,right=3cm,bottom=3.7cm,top=3.5cm]{geometry}
\usepackage{amsmath}
\usepackage{amssymb}
\usepackage{amsthm}
\usepackage{bm}
\usepackage{mathtools}
\usepackage{microtype}
\usepackage{xcolor}
\usepackage[normalem]{ulem} 
\usepackage[hidelinks]{hyperref}
\usepackage{graphicx}

\theoremstyle{plain}
\newtheorem{theorem}{Theorem}[section]
\newtheorem{lemma}[theorem]{Lemma}

\DeclareMathOperator{\spann}{span}
\newcommand{\calV}{\mathcal V}
\newcommand{\calW}{\mathcal W}

\title{Discontinuous Galerkin discretization of conservative dynamical low-rank approximation schemes for the Vlasov--Poisson equation
}
\author{
Andr\'e Uschmajew\thanks{Institute of Mathematics \& Centre for Advanced Analytics and Predictive Sciences, University of Augsburg, 86159 Augsburg, Germany} 
\and
Andreas Zeiser\thanks{Faculty 1: School of Engineering -- Energy and Information, HTW Berlin -- University of Applied
Sciences, 12459 Berlin, Germany}
}

\date{}

\begin{document}

\maketitle

\begin{abstract}
A numerical dynamical low-rank approximation (DLRA) scheme for the solution of the Vlasov--Poisson equation is presented. Based on the formulation of the DLRA equations as Friedrichs' systems in a continuous setting, it combines recently proposed conservative DLRA methods with a discontinuous Galerkin discretization. The resulting scheme is shown to ensure mass and momentum conservation at the discrete level. In addition, a new formulation of the conservative integrator is proposed based on its interpretation as a tangent space projector splitting scheme. Numerical experiments validate our approach in one- and two-dimensional simulations of Landau damping.
As a demonstration of feasibility, it is also shown that the rank-adaptive unconventional integrator can be combined with mesh adaptivity.
\end{abstract}

\section{Introduction}

The Vlasov--Poisson equation is one of the most important equations in plasma kinetics. It models the time evolution of an electron density function $f = f(t,\bm x,\bm v)$ of a collisionless plasma in $d \le 3$ spatial dimensions. It reads
\begin{equation} \label{eq:transport}
 \partial_t f + \bm v \cdot \nabla_{\bm x} f - \bm E(t, \bm x) \cdot \nabla_{\bm v} f
	= 0 \quad \textnormal{ in } \Omega = \Omega^{(x)} \times \Omega^{(v)}, \quad t \in (0,T),
\end{equation}
where $\bm E$ is the electric field dependent on the density $f$ via the Poisson equation:
\begin{equation} \label{eq:poisson}
\bm E(t, \bm x) = -\nabla_{\bm x} \Phi(t,\bm x), \quad 
	-\Delta \Phi = 1 - \int_{\Omega^{(v)}} f(t, \bm x, \bm v)\, \mathrm d \bm v.
\end{equation}
The initial condition is
\[
 f(0,\bm x, \bm v) = f_0(\bm x, \bm v).
\]
In this work we consider only spatial periodic domains $\Omega^{(x)} \subset \mathbb R^d$ and $\Omega^{(v)} = \mathbb R^d$. Otherwise the equation has to be accompanied by additional inflow boundary conditions, but this case is not considered here.

For $d=2,3$ the numerical solution of the Vlasov--Poisson equation is challenging, since a dynamical problem in $2d$ dimensions needs to be discretized. To tackle the problem, methods such as  particle methods \cite{Verboncoeur2005},  adaptive multiscale methods \cite{Deriaz2018}, and sparse grids \cite{Huang2023} have been used.

Recently, dynamical low-rank approximation (DLRA)~\cite{Koch07} has been proposed in the seminal paper~\cite{einkemmer2018a} for reducing the dimensionality based on a separation of the space and velocity variables. In DLRA the density $f$ is represented as
\begin{equation}\label{eq:dlra}
f_r(t,\bm x, \bm v) = \sum_{i=1}^r\sum_{j=1}^r X_i(t,\bm x) S_{ij}(t) V_j(t,\bm v)
\end{equation}
and its time evolution is applied to the spatial and velocity components separately. The general idea of low-rank approximation has been demonstrated to be highly effective for the Vlasov--Poisson equation in several works~\cite{Allmann-Rahn2022,guo2024a,guo2024b,Hauck2023}. A comprehensive literature survey can be found in~\cite{einkemmer2024review}.

As for DLRA methods, a recent research focus has been on the development of conservative methods, which is somewhat non-trivial to achieve. The ideas developed in~\cite{einkemmer2021} allowed for significant progress in this direction. By combining them with the rank-adaptive unconventional integrator from~\cite{CerutiKuschLubich2022}, also called basis update and Galerkin (BUG) integrator~\cite{Ceruti2024}, in~\cite{einkemmer2023} a robust numerical time stepping scheme for the factors $X_i$, $S_{ij}$ and $V_j$ in the DLRA representation~\eqref{eq:dlra} has been presented, which is mass and momentum conservative and allows the rank $r$ to be chosen adaptively. Further improvements have been achieved recently in~\cite{einkemmer2025}, where it is shown that augmentation of the spatial basis with gradients is not required for ensuring local conservation properties, hence leading to a more efficient algorithm.

The goal of the present work is to combine the conservative DLRA integrator from~\cite{einkemmer2025} with a discontinuous Galerkin (DG) discretization. This is motivated by the fact that the governing equations for the DLRA factors in the continuous setting take the form of Friedrichs' systems, as worked out in~\cite{uschmajew2024}. DG methods possess favorable numerical properties for solving such systems, and, in addition, offer the possibility for adaptive mesh refinement. A (non-adaptive) DG based DLRA algorithm for one spatial dimension has been worked out in~\cite{Ceruti2022}. With this work we make two main contributions. First, for a fixed discretization, we obtain a rank-adaptive and conservative DG-based DLRA approximation scheme for the Vlasov--Poisson equation. Second, we demonstrate how our approach of first formulating the DLRA scheme in a continuous setting and then discretizing the subproblems allows us to  exploit the capabilities of DG solvers regarding mesh adaptivity. Compared to, e.g.,~\cite{Hauck2023}, where the DLRA approach is applied to a full DG discretization of a linear kinetic equation, this approach offers much more flexibility. While the mesh adaptive version of our solver in its current form does not yet yield a conservative scheme due to the coarsening steps, we believe this approach to be a promising route for further development.

In more detail, the outline and contributions of the paper can be summarized as follows. In section~\ref{sec: equations of motion} we derive the equations of motion for the DLRA approach. We review the standard DLRA integrators in section~\ref{sec: Traditional DLRA integrators} and the required modifications for the conservative integrator as proposed in~\cite{einkemmer2021, einkemmer2023} in section~\ref{sec: Modified DLRA model}. We pay special attention to formulating the equations in a continuous setting as Friedrichs' systems in weak form, which makes it easier to enable spatially adaptive DG discretizations later. Our derivation of the equations for the conservative integrator in section~\ref{sec: modified L step} is based on a slightly different tangent space decomposition compared to~\cite{einkemmer2021,einkemmer2023,einkemmer2025}. In particular, we present a new version of the modified L-step using a suitable block QR decomposition of the DLRA representation, which avoids forming normal equations. This L-step is then included in the rank-adaptive unconventional integrator, which is introduced in section~\ref{sec: Modified unconventional integrator}.

As a side result, our special QR representation also suggests a robust projector-splitting integrator allowing for fixed velocity basis functions which has not been featured in the original works. We present this alternative integrator briefly at the end of section~\ref{sec: modified L step}. While it is not clear that this scheme is conservative by itself, it performs well in experiments and could be of some conceptual interest for further developments.

Section~\ref{sec: DG discretization} presents the main theoretical contribution. We introduce the DG discretization for the DLRA equations in sections~\ref{sec: Discontinuous Galerkin discretization for Friedrichs' systems} and~\ref{sec: Application of DG discretization to DLRA}. Special attention is paid to the approximation of derivatives and numerical fluxes at the element interfaces. As explained in~\cite{einkemmer2021,einkemmer2023,einkemmer2025}, the conservation properties of time integrators can be ensured by enforcing discrete counterparts of certain continuity equations arising in the Vlasov--Poisson system. In section~\ref{sec: Conservativion of physical invariants}, we adopt the continuous approach from~\cite{einkemmer2025} to rigorously prove these discrete continuity equations for mass and momentum conservation in our setting under the assumption that central fluxes are used in the DG discretization and the explicit Euler method is applied for time-stepping. The proof holds for arbitrary spatial dimensions. In section~\ref{sec: spatial adaptivity}, we discuss spatial adaptivity in the DG discretization, though currently in a simplified form.

Finally, in section~\ref{sec: numerical simulations} we present numerical simulations that illustrate the feasibility and advantages of our approach, and confirm the theoretical findings. As a main contribution, we present simulations for Landau damping in two spatial dimensions with conservation of mass and momentum in section~\ref{sec: Landau 2d2v}. We also demonstrate the use of spatial adaptivity for a 2d transport problem in~\ref{sec: transport adaptive DG}. While these experiments are still conducted in a simplified setting of periodic boundaries, they serve as an important proof of concept that DG methods can be successfully used for the DLRA simulation of Vlasov--Poisson equations.

For brevity, we will frequently use indices without indicating the limits. The following indices will be used: $i,j$ take values in $\{1,\dots,r\}$, $a,b$ in $\{1,\dots,m\}$, $p,q$ in $\{m+1,\dots,r\}$, and $s$ in $\{1,\ldots,d\}$. The values of other indices will be clear from the context.

\section{Conservative DLRA integrator}\label{sec: equations of motion}

In this section, we recall the conservative DLRA integrator as first presented in~\cite{einkemmer2021} and further developed in~\cite{einkemmer2023}. With the goal of an adaptive DG discretization in mind, we will, however, derive the governing equations from a weak formulation of the Vlasov--Poisson equation as given in~\cite[Sec.~76.3]{ern_2021}. Denoting the $L_2$ inner product on $\Omega = \Omega^{(x)} \times \Omega^{(v)}$ with
\[
 ( u, w )_{xv} = \int_{\Omega^{(x)}} \int_{\Omega^{(v)}}  u(\bm x, \bm v) \, w(\bm x, \bm v) \, \mathrm d (\bm x, \bm v),
\]
the weak formulation of~\eqref{eq:transport} reads: find $f \in C^1([0,T],\calW)$ such that for every $t \in (0,T)$ it holds
\begin{equation} \label{eq:weak}
( \partial_t f(t),  w )_{xv} + a(t, f(t),w) = 0
\quad \textnormal{ for all } w \in \calW,
\end{equation}
where $\calW$ is an appropriate closed subspace of $L_2(\Omega)$ (e.g.~$\calW \subseteq H^1(\Omega)$ is sufficient), and
\[
a(t,u,w) = \int_\Omega \bm v \cdot \nabla_{\bm x} u  \, w - \bm E (t,\bm x) \cdot \nabla_{\bm v} u \, w \, \mathrm d\bm x \mathrm d \bm v.
\]

Since we aim to separate the space and velocity variables, we will assume that the space $\calW$ has a tensor product structure
\[
 \calW = \calW^{(x)} \otimes \calW^{(v)} \subseteq L_2(\Omega^{(x)}) \otimes L_2(\Omega^{(v)}).
\]
We also introduce the notation
\[
( X, Y )_x = \int_{\Omega^{(x)}} X(\bm x) \, Y(\bm x) \, \mathrm d \bm x, \qquad
( V, W )_v = \int_{\Omega^{(v)}}  V(\bm v) \, W(\bm v) \, \mathrm d \bm v
\]
for $L_2$ inner products on $\Omega^{(x)}$ and $\Omega^{(v)}$.

In the DLRA ansatz, one ultimately seeks a solution $f = f_r$ of the form~\eqref{eq:dlra}. It admits several equivalent representations, such as
\begin{subequations}
\begin{align}
 f_r(t, \bm x, \bm v) &= \sum_{ij} X_i(t,\bm x) S_{ij}(t) V_j(t, \bm v) \label{eq: XSV representation} \\
 &= \sum_{j} K_j(t, \bm x) V_j(t,\bm v) \label{eq: KV representation} \\
 &= \sum_{i} X_i(t, \bm x) L_i(t, \bm v) \label{eq: XL representation},
\end{align}
\label{eq: all}
\end{subequations}
where the functions $X_1(t),\dots,X_r(t),K_1(t),\dots,K_r(t) \in \calW^{(x)}$ and $V_1(t),\dots,V_r(t), L_1(t),\dots,L_r(t) \in \calW^{(v)}$ are related through
\[
 K_j(t) = \sum_{i} X_i(t) S_{ij}(t), \qquad L_i(t) = \sum_{j} S_{ij}(t) V_j(t).
\]
Clearly, $K_j(t) \in \spann \{ X_1(t),\dots,X_r(t)\}$ and $L_i(t) \in \spann \{ V_1(t),\dots,V_r(t)\}$ for every $i$ and $j$. The $X_1(t),\dots,X_r(t)$ and $V_1(t),\dots,V_r(t)$ are usually assumed to be orthonormal systems in $\calW^{(x)}$ and $\calW^{(v)}$ for all $t$ without loss of generality, whereas this is not the case for the combined functions $K_1(t),\dots,K_r(t)$ and $L_1(t),\dots,L_r(t)$, respectively. 

\subsection{Standard DLRA integrators}\label{sec: Traditional DLRA integrators}

We first review the basic ideas of DLRA integrators before presenting the modifications proposed in~\cite{einkemmer2021,einkemmer2023,einkemmer2025} for obtaining conservative integrators in the following subsections. Using the representations~\eqref{eq: XSV representation},~\eqref{eq: KV representation} or~\eqref{eq: XL representation} at an initial time $t_0$ one can derive equations of motion separately for the matrix $S$, the functions $K_1(t),\dots,K_r(t) \in \calW^{(x)}$, and the functions $L_1(t),\dots,L_r(t) \in \calW^{(v)}$, respectively.

For example, given an initial function $f_r(t_0)$ of the considered form, equations for the time evolution of $K_1(t),\dots,K_r(t)$ in~\eqref{eq: KV representation} on a time interval $[t_0,t_1]$ are obtained by fixing the current $V_j(t_0)$ for $j=1,\dots,r$ and then restricting the weak formulation~\eqref{eq:weak} to the linear subspace
\begin{equation}\label{eq: space TV}
T_{\bm V(t_0)} \coloneqq \calW^{(x)} \otimes \spann\{V_1(t_0),\dots,V_r(t_0)\} = \left\{ \sum_j \psi_j^{(x)} \otimes V_j(t_0) : \psi_j^{(x)} \in \calW^{(x)} \right\}.
\end{equation}
This means that for $t \in [t_0,t_1]$ one seeks $f(t) = \sum_j K_j(t) \otimes V_j(t_0) \in T_{\bm V(t_0)}$ such that
\begin{equation*} \label{eq:weakTV}
( \partial_t f(t),  w )_{xv}  + a(t, f(t),w) = 0
\quad \textnormal{ for all } w \in T_{\bm V(t_0)}.
\end{equation*}
This is called a \emph{K-step}. By collecting the sought functions $K_1(t),\dots,K_r(t)$ as well as the appearing test functions $\psi_1^{(x)},\dots,\psi_r^{(x)}$ into vectors
\[
\bm K(t) = \big[ K_j(t) \big]_j, \qquad \bm \psi^{(x)} = \big[ \psi_j^{(x)} \big]_j
\]
in $(\calW^{(x)})^r$, the K-step can be written in the form of a Friedrichs' system (a coupled system of hyperbolic PDEs in weak form) as follows:
\begin{equation} \label{eq:weakK}
\big( \partial_t  \bm K(t) + \sum_{s} \mathsf A^{(x,s)} \cdot \partial_{x_s} \bm K(t) - E_s(t) \mathsf B^{(x,s)}  \cdot \bm K(t), \bm \psi^{(x)}\big)_{L_2(\Omega^{(x)}, \mathbb R^r)} = 0
 \quad \textnormal{for all } \bm \psi^{(x)} \in (\calW^{(x)})^r,
\end{equation}
with initial condition
\[
 \bm K(t_0) = \left[ \sum_i X_i(t_0) S_{ij}(t_0) \right]_j
\]
and suitable matrices $\mathsf A^{(x,s)}$ and $\mathsf B^{(x,s)}$; cf.~\cite{uschmajew2024}. We present these matrices for the conservative integrator in section~\ref{sec: tangent space decomposition}.

Equations for the velocity functions $L_1(t),\dots,L_r(t)$ in the representation~\eqref{eq: XL representation} on a time interval $[t_0,t_1]$ are obtained in an analogous way by fixing $X_1(t_0),\dots,X_r(t_0)$ instead, and restricting the weak formulation~\eqref{eq:weak} to the linear space
\begin{equation}\label{eq: space TX}
T_{\bm X(t_0)} \coloneqq \spann\{ X_1(t_0),\dots,X_r(t_0) \} \otimes \calW^{(v)} = \left\{ \sum_i X_i(t_0) \otimes \psi_i^{(v)}  \colon \psi_i^{(v)} \in \calW^{(v)} \right\}
\end{equation}
instead. This is called the \emph{L-step} and also results in a Friedrichs' system, but this time for the vector function
\[
\bm L(t) = \big[ L_i(t) \big]_i \in (L_2(\Omega^{(v)}))^r
\]
and with test functions of the same form. The system will be presented for the conservative integrator in section~\ref{sec: tangent space decomposition}.

Finally, in the \emph{S-step} one only considers the time evolution of the coefficient matrix $S(t) = [S_{ij}(t)]$ in~\eqref{eq: XSV representation} with the basis functions $X_i(t_0)$ and $V_j(t_0)$ being fixed. One can do so by restricting the weak formulation~\eqref{eq:weak} to the \emph{finite-dimensional} space
\begin{equation}\label{eq: space TXV}
T_{\bm X(t_0),\bm V(t_0)} \coloneqq \spann\{ X_1(t_0),\dots,X_r(t_0) \} \otimes \spann \{ V_1(t_0),\dots,V_r(t_0) \}.
\end{equation}
This gives an ODE of the form
\begin{equation} \label{eq:weakS}
\partial_t S(t) + \sum_s \Big[ \mathsf B^{(v,s)} S(t) \big(\mathsf A^{(x,s)}\big)^T
	+ \mathsf A^{(v,s)} S(t) \big(\mathsf B^{(x,s)}\big)^T \Big] = 0
\end{equation}
for the matrix $S(t)$ with initial condition $S(t_0)$ and suitable matrices $\mathsf A^{(x/v,s)}, \, \mathsf B^{(x/v,s)}$.

Based on the above restriction of the time evolution to particular subspaces, the so-called projector splitting integrator~\cite{lubich2014,einkemmer2018a} can be derived. It combines the three steps for updating the factors from a time point $t_0$ to a time point $t_1$ in the order ``KSL'' as outlined below. The name \emph{projector splitting} stems from the fact that the integrator realizes in first order a time step on the subspace $T_{\bm V(t_0)} + T_{\bm X(t_0)} \subset \calW^{(x)} \otimes \calW^{(v)}$, which is the tangent space to the manifold of functions of the form~\eqref{eq: XSV representation} at the current point $f_r(t_0)$ (under the mild technical assumption that $S(t_0)$ has full rank $r$). Since the spaces $T_{\bm V(t_0)}$ and $T_{\bm X(t_0)}$ used for the K-step and L-step intersect precisely in the space $T_{\bm X(t_0), \bm V(t_0)}$ defined in~\eqref{eq: space TXV}, this intersection space would be considered ``twice'', both in the K- and the L-step. Therefore, it is necessary to perform the S-step in $T_{\bm X(t_0), \bm V(t_0)}$ ``backwards'' to compensate for this. The individual steps in the following \emph{KSL scheme}, however, also always account for the updated basis functions from the previous step. We refer to~\cite{lubich2014} for why the S-step is best executed between the K- and L-step.

\subsubsection*{Projector splitting integrator: KSL scheme}

In detail, the KSL scheme for a time step from $t_0$ to $t_1$ proceeds as follows.

\begin{enumerate}
 \item
 Perform a K-step on $T_{\bm V(t_0)}$ by solving the Friedrichs' systems~\eqref{eq:weakK} on $[t_0,t_1]$ for new functions $K_1(t_1),\dots,K_r(t_1) \in \calW^{(x)}$. Choose an orthonormal basis $X_1(t_1),\dots,X_r(t_1)$ for a subspace containing $K_1(t_1),\dots,K_r(t_1)$
together with coefficients $\tilde S_{ij}$ such that $K_j(t_1) = \sum_{i} \tilde S_{ij} X_i(t_1)$.
 \item
 Take $S(t_0) = \tilde S$ as start point for a ``backward'' S-step on $T_{\bm X(t_1), \bm V(t_0)}$, that is, solve the ODE~\eqref{eq:weakS} with $\partial S(t) - \ldots {}$ instead of $\partial S(t) + \ldots {}$ and using the already updated $X_1(t_1),\dots,X_r(t_1)$ instead of $X_1(t_0),\dots,X_r(t_0)$. This gives a new matrix $\hat S$ at time $t_1$.
 \item
 Perform an L-step on $T_{\bm X(t_1)}$ by solving the corresponding Friedrichs' system for new functions $L_1(t_1),\dots,L_r(t_1) \in \calW^{(v)}$ with recomputed initial conditions $L_i(t_0) = \sum_j \hat S_{ij} V_j(t_0)$. Choose an orthonormal basis $V_1(t_1),\dots,V_r(t_1)$ for a subspace containing of $L_1(t_1),\dots,L_r(t_1)$ together with, and obtain coefficients $S_{ij}(t_1)$ such that $L_i(t_1) = \sum_j S_{ij}(t_1) V_j(t_1)$.
\end{enumerate}
The result at time $t_1$ is the function
\[
f_r(t_1,\bm x, \bm v) = \sum_{ij} X_i(t_1, \bm x) S_{ij}(t_1) V_j(t_1, \bm v).
\]

The projector splitting integrator has a clear geometric interpretation and favorable numerical properties but has the practical disadvantage that the target rank $r$ needs to be selected beforehand and remains fixed over time. Since the rank at any time should be large enough to ensure sufficient accuracy while remaining small enough to keep the computations feasible, an adaptive procedure is more desirable. The rank-adaptive unconventional DLRA integrator proposed in~\cite{CerutiKuschLubich2022} addresses this issue using a different ordering of steps and basis augmentation. It will be discussed in section~\ref{sec: Modified unconventional integrator}.

\subsection{Conservative DLRA integrator}\label{sec: Modified DLRA model}

The exact solution to the Vlasov-Poisson equation conserves several physically relevant quantities such as mass, momentum and energy. Ensuring this in DLRA simulations requires nontrivial modifications and has been addressed only recently by Einkemmer and Joseph in~\cite{einkemmer2021}. Consider for example the mass
\[
m(t) = \int_{\Omega^{(x)}} \int_{\Omega^{(v)}}  f(t,\bm x,\bm v) \, \mathrm d (\bm x, \bm v).
\]
Employing a DLRA model~\eqref{eq:dlra} for $f$ would require a space $\calW^{(v)} \subseteq L_2(\mathbb R^d) \cap L_1(\mathbb R^d)$ for the velocity dependence to ensure that the mass is well-defined. (For the space $\calW^{(x)} \subseteq L_2(\Omega^{(x)})$ there is no problem since the periodic spatial domain $\Omega^{(x)}$ is assumed to be bounded.) On the other hand, since we are considering weak formulations, quantities of interest can only be assessed by testing with suitable functions. For mass, this would be the constant function $(\bm x, \bm v) \mapsto 1$ since $m(t) = (f(t), 1)_{xv}$. For momentum and kinetic energy, this would be the functions $(\bm x, \bm v) \mapsto v_s$ and $(\bm x, \bm v) \mapsto \| \bm v \|^2$, respectively. All these functions are clearly not in $L_2(\mathbb R^d) \cap L_1(\mathbb R^d)$ either. As explained in~\cite{einkemmer2021}, choosing $\calW^{(v)}$ as a space of functions with bounded support is not recommended to address this problem since it implicitly restricts the velocity variable to a bounded domain which is not the desired physical model.

To resolve the conflict, Einkemmer and Joseph propose the following modifications. First, the space of velocity functions is enlarged to actually accommodate for the desired test functions. Specifically, one assumes
\[
V_j \in \calW^{(v)}_{\omega} \subseteq L_2(\Omega^{(v)}; \omega),
\]
where $L_2(\Omega^{(v)}; \omega)$ is a Hilbert space of \emph{weighted} $L_2$ functions equipped with inner product and norm
\[
(V,W)_{v,\omega} \coloneqq \int_{\Omega^{(v)}} V(\bm v) W(\bm  v) \omega(\bm v) \, \mathrm d \bm v, \qquad \| V \|_{v,\omega} \coloneqq (V,V)_{v,\omega}^{1/2},
\]
with a strictly positive and sufficiently fast decaying weight function $\omega$ ensuring that $L_2(\Omega^{(v)};\omega)$ contains the functions $1$, $v_1,\dots,v_d$, and $\| \bm v \|^2$. A somewhat natural choice is the Gaussian weight function
\begin{equation}\label{eq: Gauss weight}
\omega(\bm v) = \exp(- \| \bm v \|^2/2).
\end{equation}

Second, in order to be able to test with $L_2(\Omega^{(v)};\omega)$ functions in the ordinary $L_2$ inner product, the low-rank model for $f$ is modified to explicitly carry the weight function $\omega$, that is, instead of~\eqref{eq:dlra} one now seeks $f = f_r$ of the form
\begin{equation}\label{eq:weighted dlra}
 f_r (t,\bm x, \bm v) = \omega(\bm v) \sum_{ij} X_i(t,\bm x) S_{ij}(t) V_j(t,\bm v)
\end{equation}
where $X_1(t),\dots,X_r(t) \in \calW^{(x)}$ as before, but now $V_1(t),\dots,V_r(t) \in \calW^{(v)}_{\omega}$ are allowed.

Third, since ultimately only test spaces of the form $T_{\bm V(t)}$, $T_{\bm X(t)}$ or $T_{\bm X(t),\bm V(t)}$ (cf.~\eqref{eq: space TV},~\eqref{eq: space TX} and~\eqref{eq: space TXV}) will be used in the numerical integrator, it will turn out necessary to ensure that at every time the space spanned by the velocity functions $V_1(t),\dots,V_r(t)$ in the weighted DLRA representation~\eqref{eq:weighted dlra} contains the functions $\bm v \mapsto 1$, $\bm v \mapsto v_s$ and $\bm v \mapsto \| \bm v \|^2$. This is achieved by simply fixing some of the factors $V_j$ explicitly and not allowing them to vary. Following the notation in~\cite{einkemmer2021}, the first $m$ functions $V_j$ will be kept fixed and denoted as
\[
U_a = V_a, \quad a = 1,\dots,m,
\]
such that $\spann\{U_1,\dots,U_m\} \subseteq \calW^{(v)}_{\omega}$ contains the desired functions. In particular, $m \ge 1$ for mass, $m \ge 1 + d$ for mass and momentum conservation, respectively. The remaining $V_j$ are still variable and will be denoted by $W_p$ instead:
\[
W_p = V_p, \quad p = m+1,\dots,r.
\]
The DLRA model set for $f_r$ then ultimately is
\begin{equation}\label{eq:Mr weighted}
\begin{split}
\mathcal M_{r,\omega} = \Big\{ \, f_r = \omega \sum_{ia} &S_{ia} X_i \otimes  U_a + \omega \sum_{ip} S_{ip} X_i \otimes W_p  : \\ &S \in \mathbb R^{r \times r}, \, X_i \in \calW^{(x)},  \, W_p \in \calW^{(v)}_{\omega} \text{ and } (U_a, W_p)_{v,\omega} = 0 \, \Big\}.
\end{split}
\end{equation}
The orthogonality relations $(U_a, W_p)_{v,\omega} = 0$ in $\calW^{(v)}_{\omega}$ should hold for all pairs $a,p$ and ensure that the $W_p$ are linearly independent from the $U_a$ to avoid trivial cases. When working with functions from $\mathcal M_{r,\omega}$ we will, without loss of generality, additionally assume that the $X_1, \dots X_r$ are pairwise orthonormal in $L_2(\Omega^{(x)})$ and the $W_{m+1},\dots,W_r$ are pairwise orthonormal in $L_2(\Omega^{(v)};\omega)$. We also assume that the fixed functions $U_1,\dots,U_m$ are pairwise orthonormal in $L_2(\Omega^{(v)};\omega)$.

As a fourth and final modification, Einkemmer and Joseph propose to consider the time-dependent variational principle for the time evolution on $\mathcal M_{r,\omega}$ in the form of a Petrov--Galerkin formulation as
\begin{equation} \label{eq:weak_dlra}
\big( \partial_t f_r(t),  \frac{1}{\omega} w \big)_{xv}  + a\big(t, f_r(t),\frac{1}{\omega} \, w\big) = 0
\textnormal{\quad for all } w \in T_{f_r(t)} \mathcal M_{r,\omega},
\end{equation}
where $T_{f_r(t)} \mathcal M_{r,\omega}$ denotes the tangent space to the ``manifold'' $\mathcal M_{r,\omega}$. We will not rigorously investigate manifold properties of this set, but it is clear from the multilinearity of the representation in~\eqref{eq:Mr weighted} that the tangent space at $f_r \in \mathcal M_{r,\omega}$ takes the form
\begin{equation} \label{eq:TMr}
\begin{split}
T_{f_r} \mathcal M_{r,\omega} =
	\Big\{ \,  \omega \sum_{ij}  \big(&\dot X_i S_{ij} + X_i \dot S_{ij} \big) \otimes V_j
		+  \omega  \sum_{ip} S_{ip} X_i \otimes \dot W_p : \\
	& \dot S \in \mathbb R^{r\times r}, \, \dot X_i \in \calW^{(x)}, \,
	\dot W_p \in \calW^{(v)}_{\omega} \text{ and } (U_a, \dot W_p )_{v,\omega} = 0 \, \Big\}.
\end{split}
\end{equation}
As one can see, the weight function $\omega$ introduced for $f_r$ in~\eqref{eq:weighted dlra} carries over to the tangent space at $f_r$, which would result in a squared weight when testing with $f_r$ itself. Dividing by $\omega$ in~\eqref{eq:weak_dlra} hence removes $\omega$ from the tangent space and, in particular, ensures that the functions $(\bm x,\bm v) \mapsto U_a(\bm v)$ are actually available as test functions as desired (by choosing $\dot S = 0$, $\dot W_p = 0$ for all $p$, and $\dot X_i = \alpha_i 1$ such that $\sum_{i} \alpha_i S_{ij} = \delta_{ja}$ in~\eqref{eq:TMr}, which is possible if $S$ has full rank $r$).

As in the standard DLRA integrators, the formulation~\eqref{eq:weak_dlra} will next be further restricted to linear subspaces within the tangent space. However, since some of the $V_a = U_a$ are fixed for $a=1,\dots,m$, the corresponding K-, L- and S-steps require a more careful definition. We discuss the corresponding tangent space decomposition in the next section.

\subsection{Tangent space decomposition and modified projector splitting}\label{sec: tangent space decomposition}

In~\cite{einkemmer2021} the equations of motion for $X_i$, $S$ and $V_j$ factors have been derived, but an inversion of $S$ was required for the updates of $V_j$. As an extension of their ideas, we derive a decomposition of the tangent space $T_{f_r}  \mathcal M_{r,\omega}$ into orthogonal subspaces which allows to avoid the inversion of $S$, in the same spirit as in the original projector splitting integrator~\cite{lubich2014}. For that, consider a particular low-rank representation of a given $f_r \in \mathcal M_{r,\omega}$ tailored to the fact that the first $m$ velocity basis functions are fixed. We can assume that $f_r$ is initially given in the form $f_r = \omega \sum_{j} K_j \otimes V_j$. We we can then apply a suitable orthogonal transformation to the functions $K_1,\dots,K_r$ using a variant of (block) QR decomposition, to obtain a new set of pairwise orthonormal functions $X_1,\dots,X_r \in \calW^{(x)}$ such that $K_j = \sum_i X_i S_{ij}$ as usual, but now with a block triangular matrix $S$ of the form
\begin{equation}\label{eq:triangular S}
S = \begin{pmatrix} [S_{ab}] & 0_{m \times (r-m)} \\ [S_{pb}] & [S_{pq}]  \end{pmatrix}.
\end{equation}
Hence, without loss of generality, we can write $f_r$ as
\begin{equation}\label{eq:QR representation}
f_r = \omega \, \sum_{ab} S_{ab} X_a \otimes U_b + \omega \sum_{pb} S_{pb} Z_p \otimes U_b + \omega  \sum_{pq} S_{pq} Z_p \otimes W_q
\end{equation}
where we introduced the notation
\[
Z_p = X_p, \quad p=m+1,\dots,r.
\]

Taking this representation of $f_r$, and additionally assuming that the matrix $S$ has full rank $r$, it is easy to deduce from~\eqref{eq:TMr} that the tangent space at $f_r$ can be decomposed into
\[
T_{\bm V,\omega} = \Big\{ \, \omega \sum_j \psi_j^{(x)} \otimes V_j : \psi_j \in \calW^{(x)} \, \Big\}
\]
and (using that $S_{ap} = 0$ in the second sum of~\eqref{eq:TMr})
\begin{equation}\label{eq: space Tz}
T_{\bm Z, \omega} = \Big\{ \, \omega \sum_p Z_p^{} \otimes \widehat{\psi}_p^{(v)} : \widehat{\psi}_p^{(v)} \in \calW^{(v)}_{\omega}, \, ( U_a, \widehat{\psi}_p^{(v)} )_{v,\omega} = 0 \, \Big\},
\end{equation}
such that
\begin{equation}\label{eq: T_V+T_Z}
T_{f_r} \mathcal M_{r,\omega} =  T_{\bm V, \omega} + T_{\bm Z,\omega}.
\end{equation}
However, this sum is not direct, as the spaces intersect in
\begin{equation}\label{eq:space T_ZV}
T_{\bm Z,\bm V, \omega} = T_{\bm Z, \omega} \cap T_{\bm V, \omega} = \Big\{ \, \omega \sum_{pq} \Sigma_{pq} Z_p \otimes W_p : \Sigma_{pq} \in \mathbb R \, \Big\}.
\end{equation}
Based on the tangent space decomposition~\eqref{eq: T_V+T_Z}--\eqref{eq:space T_ZV}, we design a projector splitting integrator as well as a variant of the unconventional integrator. For this, we next derive the corresponding equations of motion as Friedrichs' systems on subspaces $T_{\bm V,\omega}$ (modified K-step), $T_{\bm Z, \omega}$ (modified L-step) and $T_{\bm Z,\bm V, \omega}$ (modified S-step). Note that these spaces are defined solely by the basis sets $X_1,\dots,X_r$ and $W_{m+1},\dots,W_r$ and hence can be defined even if $S$ is not invertible.

\subsubsection*{Modified K-step}\label{sec: modified K-step}

The time dependent solution of \eqref{eq:weak_dlra} restricted to $T_{\bm V,\omega}$ is a function
\[
f_r(t) = \omega \sum_j K_j(t) \otimes V_j.
\]
The test functions take the form
\[
w  = \omega \sum_{i} \psi_i^{(x)} \otimes V_i^{},
\]
where $\psi_1^{(x)},\dots,\psi_r^{(x)} \in \calW^{(x)}$ are arbitrary. Substituting this into \eqref{eq:weak_dlra} gives
\begin{equation*}
\begin{split}
& \big(
	 \omega \sum_{j} \partial_t K_j(t) \otimes V_j, \,
		\sum_{i} \psi_i^{(x)} \otimes V_i \big)_{xv} \\
& \quad + \sum_s
	\Big( v_s \partial_{x_s}  \big( \omega \sum_{j} K_j(t) \otimes V_j \big)
		- E_s \partial_{v_s}  \big( \omega \sum_{j} K_j(t) \otimes V_j \big), \, \sum_{i} \psi_i^{(x)} \otimes V_i^{} \Big)_{xv}
			= 0.
\end{split}
\end{equation*}
Using the orthogonality relations $(\omega V_j,V_i)_v = \delta_{ij}$ in the first term and rearranging, it follows that
\begin{equation} \label{eq:Kindex}
\begin{split}
\sum_i \big(  \partial_t K_i(t) , \, \psi_i^{(x)} \big)_x
	+ \sum_s \Bigg[ &
		\sum_{i} \Big( \sum_j \big(\omega v_s V_j, V_i\big)_v \partial_{x_s} K_j(t), \, \psi_i^{(x)} \Big)_x \\
&+ \sum_{i} \Big( \sum_j -E_s \big(\partial_{v_s} (\omega  V_j), V_i\big)_v K_j(t), \, \psi_i^{(x)} \Big)_x
		\Bigg] = 0.
\end{split}
\end{equation}
As in the standard DLRA integrator, we write this equation in terms of the vector-valued functions $\bm K(t) = \big[ K_j(t) \big]_j$ and $\bm \psi^{(x)} = \big[ \psi_j^{(x)} \big]_j$ in $(\calW^{(x)})^r$. Specifically, defining the matrices
\begin{equation} \label{eq:Axs_Bxs}
\mathsf A^{(x,s)} = \Big[ \big( v_s V_j, V_i \big)_{v,\omega} \Big]_{ij}, \qquad
\mathsf B^{(x,s)} = \Big[  \big(\partial_{v_s}(\omega V_j), V_i \big)_v \Big]_{ij},
\end{equation}
eq.~\eqref{eq:Kindex} can be rewritten as
\begin{equation}\label{eq:Kvector}
\big( \partial_t \bm K(t), \bm \psi^{(x)}\big)_x
	+ \sum_s
		\big( \mathsf A^{(x,s)} \partial_{x_s} \bm K(t) - E_s(\cdot) \mathsf B^{(x,s)} \bm K(t), \, \bm \psi^{(x)}\big)_x
			 = 0 \quad \text{for all } \bm \psi^{(x)} \in (\calW^{(x)})^r,
\end{equation}
which is the weak form of a Friedrichs' system.

\subsubsection*{Modified L-step}\label{sec: modified L step}

For a modified L-step, we first write the initial function $f_r = f_r(t_0) \in \mathcal M_{r,\omega}$ based on~\eqref{eq:QR representation} in the form
\begin{equation*}
 f_r(t_0) =
 	\omega \sum_{ib} S_{ib} X_i \otimes U_b
 	+ \sum_{p} Z_p \otimes \widehat{L}_{\omega,p}(t_0)
\end{equation*}
where we have incorporated the weight function $\omega$ into $\widehat L_{\omega,p}$ as follows:
\[
\widehat{L}_{\omega,p}(t_0) = \omega L_p(t_0) = \omega \sum_q S_{pq} W_q \in \calW^{(v)}.
\]
We remark that due to our particular representation~\eqref{eq:QR representation}, this formula slightly differs from the corresponding formula in~\cite[Eq.~(6)]{einkemmer2023}, where products of $S$ and its transpose are required for forming the new $L_p$ factors.
Recall that  $(U_a,\widehat{L}_{\omega,p}(t_0))_v = 0$ for all $a$ and $p$. The goal is the time evolution of $\widehat{L}_{\omega, p}(t)$ under this orthogonality constraint which will yield
\[
f_r(t) = \omega \sum_{i} S_{ib} X_i \otimes U_b
 	+ \sum_{p} Z_p \otimes \widehat{L}_{\omega,p}(t)
\]
i.e.~with the same ``fixed'' first part. The second sum in the above expression is a function in the subspace $T_{\bm Z,\omega}$ of $T_{f_r(t_0)} \mathcal M_{r,\omega}$ as defined in~\eqref{eq: space Tz}. Hence, equations for the time evolution of $\widehat{L}_{\omega,p} (t)$ are obtained by requiring the weak formulation~\eqref{eq:weak_dlra} only for test functions from $T_{\bm Z,\omega}$, that is, test functions of the form
\[
w = \omega \sum_p Z_p^{} \otimes \widehat{\psi}_p^{(v)}
\]
with $\widehat{\psi}_p^{(v)} \in \calW^{(v)}_{\omega}$ and $(U_a,\widehat{\psi}_p^{(v)} )_{v,\omega} = 0$. Since for such test functions one has $(\partial_t f_r(t),  \frac{1}{\omega} w \big)_{xv} = (\partial_t \sum_{p} Z_p \widehat{L}_{\omega,p}(t),  \frac{1}{\omega} w)_{xv}$, this leads to the formulation
\[
\big( \partial_t \sum_{p} Z_p \widehat{L}_{\omega,p}(t),  \frac{1}{\omega} w \big)_{xv}  + a \big(t, f_r(t),\frac{1}{\omega} \, w \big) = 0
\]
for all $w$ of the given form.

In detail, using orthogonality relations this means
\begin{align*}
&\sum_p \big(  \partial_t \widehat L_{\omega, p}(t), \, \widehat \psi_p^{(v)} \big)_v
	+ \sum_s \Bigg[
		\sum_{pq} \big(
			(-E_s Z_p, Z_q)_x \partial_{v_s} \widehat L_{\omega,p}(t)
			+ (\partial_{x_s} Z_p, Z_q)_x  v_s \widehat L_{\omega,p}(t), \, \widehat \psi_q^{(v)} \big)_v \Bigg]
		\\
&= -\sum_s \Bigg[
	\sum_{iaq} \big(
		(-E_s X_i, Z_q)_x S_{ia} \partial_{v_s}(\omega U_a)
			+ (\partial_{x_s} X_i, Z_q)_x  S_{ia}  v_s \omega U_a , \, \widehat \psi_q^{(v)}
		\big)_v
	  \Bigg].
\end{align*}
Using the vector-valued function
\[
\widehat{\bm L}_\omega(t) = [\widehat L_{\omega, p}(t)]_p, \qquad \widehat{\bm \psi}^{(v)} = \big[ {\widehat \psi_p}^{(v)} \big]_p
\]
in $(\calW^{(v)}_{\omega})^{r-m}$ this can be rewritten as
\begin{equation} \label{eq:Lhat}
\begin{split}
\big( \partial_t \widehat{\bm L}_\omega(t), \widehat{\bm \psi}^{(v)} \big)_v
	+ &\sum_s
		\big( \widehat{\mathsf A}^{(v,s)} \partial_{v_s} \widehat{\bm L}_\omega(t)
			+  v_s  \widehat{\mathsf B}^{(v,s)} \widehat{\bm L}_\omega(t), \, \widehat{\bm \psi}^{(v)}\big)_v \\
= - &\sum_s
		\big( \bar{\mathsf A}^{(v,s)} \bar S \partial_{v_s}( \omega \bm U)
			+  v_s  \bar{\mathsf B}^{(v,s)} \bar S \omega \bm U, \, \widehat{\bm \psi}^{(v)}\big)_v \quad \text{for all } \widehat{\bm \psi}^{(v)} \in (\calW^{(v)}_{\omega})^{r-m},
\end{split}
\end{equation}
with $\bar S = [S_{ia}]_{ia}$, $\bm U = [U_a]_a$, and
\begin{equation} \label{eq:hatbarv}
\widehat{\mathsf A}^{(v,s)} = \Big[\mathsf A^{(v,s)}_{pq}\Big]_{pq}, \quad
	\bar{\mathsf A}^{(v,s)} = \Big[\mathsf A^{(v,s)}_{pj}\Big]_{pj}, \quad
\widehat{\mathsf B}^{(v,s)} = \Big[\mathsf B^{(v,s)}_{pq}\Big]_{pq}, \quad
	\bar{\mathsf B}^{(v,s)} = \Big[\mathsf B^{(v,s)}_{pj}\Big]_{pj},
\end{equation}
where
\begin{equation} \label{eq:Avs_Bvs}
\mathsf A^{(v,s)} = \Big[ \big(-E_s X_j, X_i \big)_x \Big]_{ij}, \qquad
	\mathsf B^{(v,s)} = \Big[  \big(\partial_{x_s} X_j, X_i \big)_x \Big]_{ij}.
\end{equation}
Hence,~\eqref{eq:Lhat} is an \emph{inhomogeneous} Friedrichs' system for $r-m$ functions $\widehat{L}_{\omega,p}(t)$ in $\calW^{(v)}_{\omega}$ to be solved over the time interval $[t_0,t_1]$. By construction, these will satisfy $(\widehat{L}_{\omega,p}(t), U_a)_v = 0$ for all $t \in [t_0,t_1]$. Eventually, we obtain $L_p(t_1) = \omega^{-1} \, \hat L_p(t_1)$. We highlight that~\eqref{eq:Lhat} represents an orthogonal projection of the system onto the space $T_{\bm Z, \omega}$ and therefore differs from the formulation of the modified L-step in~\cite[Eq.~(11)]{einkemmer2021} which requires the inversion of the matrix $S$.

\subsubsection*{Modified S-step}

The formulation of the modified S-step also starts out from the particular representation~\eqref{eq:QR representation}, this time written as
\begin{equation} \label{eq:S_fr}
f_r(t_0) = \omega \sum_{ia} S_{ia} X_i \otimes U_a  + \omega \sum_{pq} S_{pq}(t_0) Z_p \otimes W_q.
\end{equation}
The goal is to evolve only $S_{pq}(t)$ by restricting the weak formulation~\eqref{eq:weak_dlra} to the finite-dimensional subspace $T_{\bm Z, \bm V, \omega}$ in~\eqref{eq:space T_ZV}, which is orthogonal to the first part of the representation~\eqref{eq:S_fr}. Similar to the modified L-step, this leads to an inhomogeneous system. Specifically, writing the test functions as
\[
w = \omega \sum_{p'q'} \Sigma_{p'q'} Z_{p'} \otimes W_{q'},
\]
where $p',q'$ run from $m+1,\dots,r$, testing with $w$ first gives
\begin{equation*}
\begin{split}
& \Big(\omega  \sum_{pq} \partial_t S_{pq}(t) Z_p \otimes  W_q,\, \sum_{p'q'} \Sigma_{p'q'} Z_{p'} \otimes W_{q'} \Big)_{xv} \\
& \quad + \sum_s
	\Big( v_s \partial_{x_s}  \big( \omega \sum_{ij} S_{ij}(t) X_i \otimes V_j \big)
		- E_s \partial_{v_s}  \big( \omega \sum_{ij} S_{ij}(t) X_i \otimes V_j  \big),
			\, \sum_{p'q'} \Sigma_{p'q'} Z_{p'} \otimes W_{q'} \Big)_{xv}	= 0.
\end{split}
\end{equation*}
Using the orthogonality relations $(Z_p,Z_{p'})_x = \delta_{pp'}$ and $(W_q,W_{q'})_{v,\omega} = \delta_{qq'}$ this reduces to
\begin{equation*}
\begin{split}
& \sum_{pq} \partial_t S_{pq}(t) \Sigma_{pq}
+  \sum_s \sum_{\substack{pp'\\qq'}} \Big[
	 {\mathsf B}^{(v,s)}_{p'p} S_{pq}(t) {\mathsf A}^{(x,s)}_{q'q}  +
		{\mathsf A}^{(v,s)}_{p'p} S_{pq}(t) {\mathsf B}^{(x,s)}_{q'q} \Big] \Sigma_{p'q'} \\
&= -  \sum_s \sum_{\substack{ip'\\ bq'}} \Big[
	 {\mathsf B}^{(v,s)}_{p'i} S_{ib}(t) {\mathsf A}^{(x,s)}_{q'b}  +
		{\mathsf A}^{(v,s)}_{p'i} S_{ib}(t) {\mathsf B}^{(x,s)}_{q'b} \Big] \Sigma_{p'q'},
\end{split}
\end{equation*}
where we used the splitting \eqref{eq:S_fr}. Now, since $\Sigma$ is an arbitrary matrix, this gives
\begin{equation}\label{eq:Spartial}
\begin{split}
\partial_t \widehat S(t) &+ \sum_s \Big[
	\widehat{\mathsf B}^{(v,s)} \widehat S(t) \big(\widehat{\mathsf A}^{(x,s)}\big)^T
		+ \widehat{\mathsf A}^{(v,s)} \widehat S(t) \big(\widehat{\mathsf B}^{(x,s)}\big)^T\Big]  \\
&= -\sum_s \Big[
	\bar{\mathsf B}^{(v,s)} \bar S(t) \big(\bar{\mathsf A}^{(x,s)}\big)^T
		+ \bar{\mathsf A}^{(v,s)} \bar S(t) \big(\bar{\mathsf B}^{(x,s)}\big)^T\Big]
\end{split}
\end{equation}
for $\widehat S(t) = \big[S_{pq}(t)\big]_{pq}$ and $\bar S(t) = \big[S_{ia}(t)\big]_{ia}$, where the matrices are submatrices of~\eqref{eq:Axs_Bxs}:
\[
\widehat{\mathsf A}^{(x,s)} = \Big[\mathsf A^{(x,s)}_{pq}\Big]_{pq}, \quad
	\bar{\mathsf A}^{(x,s)} = \Big[\mathsf A^{(x,s)}_{pb}\Big]_{pb}, \quad
\widehat{\mathsf B}^{(x,s)} = \Big[\mathsf B^{(x,s)}_{pq}\Big]_{pq}, \quad
	\bar{\mathsf B}^{(x,s)} = \Big[\mathsf B^{(x,s)}_{pb}\Big]_{pb}.
\]
This is hence an inhomogenous ODE for the submatrix $\widehat S(t)$.

\subsubsection*{A modified projector splitting integrator}\label{sec: Modified projector splitting integrator}

As already mentioned, the tangent space decomposition~\eqref{eq: T_V+T_Z}--\eqref{eq:space T_ZV} suggests a projector splitting integrator based on the modified K-, S- and L-steps, which we present here for completeness. Although numerical experiments show that this scheme technically works (see section~\ref{sec: Landau1d1v}), keeping the functions $U_a$ in the bases by itself might not yet ensure the desired conservation properties. To guarantee this, the rank-adaptive unconventional integrator presented further below is necessary.

Starting from a time point $t_0$ and a current representation of $f_r(t_0)$ in the form~\eqref{eq:QR representation} we denote $V_j^0 = V_j(t_0)$. In particular, $U_a = V^0_a$ for $a=1,\dots,m$ and $W_p^0 = V_p^0$ for $p = m+1,\dots,r$. The numerical computation of $f_r(t_1)$ at a next time step $t_1=t_0 + \Delta t$ consists of three steps:

\begin{enumerate}
\item \label{enum:stepK}
(Modified K-step) Denoting $\bm V^0 = [V_j^0]_{j=1,\dots,r}$ and $\bm K^0 = [K_j(t_0)]_{j=1,\dots,r}$, solve the system~\eqref{eq:weak_dlra} restricted to the subspace $T_{\bm V^0,\omega}$ on the time interval $[t_0,t_1]$ by solving the Friedrichs' system~\eqref{eq:Kvector} with initial condition $\bm K^0$. At time $t_1$, one obtains
\[
\tilde f = \omega \, \sum_j K_j^1 \otimes V_j^0 \in T_{\bm V^0}.
\]
Restore the form~\eqref{eq:QR representation} by finding an orthonormal system $\bm X^1 = [X_1^1, \ldots, X_r^1]$ such that
\[
K_j^1 = \sum_i X_i^1 \tilde S_{ij}
\]
and $\tilde S$ is of the block triangular form~\eqref{eq:triangular S}. Let $Z^1_p = X^1_p$.
\item \label{enum:stepS}
(Modified S-step) Solve the system~\eqref{eq:weak_dlra} restricted to the subspace $T_{\bm Z^1,\bm V^0, \omega}$ ``backward'' on the time interval $[t_0,t_1]$ with initial condition $\tilde f$. This is achieved by solving~\eqref{eq:Spartial} with $\partial_t \widehat{S} - \ldots$ (instead of $\partial_t \widehat{S} + \ldots$)  on $[t_0,t_1]$ with initial value $\tilde S$ from step 1 and corresponding matrices $\widehat{\mathsf{A}}^{(x,s)}$, $\bar{\mathsf{A}}^{(x,s)}$, $\widehat{\mathsf{B}}^{(x,s)}$, $\bar{\mathsf{B}}^{(x,s)}$ constructed from $\bm X^1$ and $\bm V^0$. At time $t_1$ one obtains
\[
\hat f = \omega \, \sum_{ab} \tilde S_{ab} X_a^1 \otimes U_b + \omega \, \sum_{pb} \tilde S_{pb} Z_p^1 \otimes U_b + \omega \, \sum_{pq} \widehat S_{pq} Z_p^1 \otimes W_q^0.
\]
\item \label{enum:stepL}
(Modified L-step)
Solve the system~\eqref{eq:weak_dlra} restricted to the subspace $T_{\bm Z^1,\omega}$ on the time interval $[t_0,t_1]$ with initial condition $\hat f$ by solving the modified L-step~\eqref{eq:Lhat} with the data computed from $\hat f$. At time $t_1$ one obtains
\[
\bar f_r = \omega \, \sum_i X_i^1 \otimes L_i^1 + \omega \, \sum_{p} Z_p^1 \otimes L_p^1
\]
which is also taken as the final solution at time $t_1$ after recomputing it into the form~\eqref{eq:QR representation}.
\end{enumerate}

\subsection{Rank-adaptive unconventional integrator}\label{sec: Modified unconventional integrator}

The rank $r$ in the DLRA model needs to be large enough to ensure accuracy of the solution, but still (if possible) small enough to keep the model efficient. Since determining a suitable value of $r$ is necessary, the missing rank-adaptivity of the projector splitting integrator poses a practical disadvantage. The rank-adaptive integrator proposed in~\cite{CerutiKuschLubich2022} addresses this issue by first computing the K- and L-step independently, and then using basis augmentation in the S-step. It has been adopted to the conservative setup with fixed velocity functions in~\cite{einkemmer2023,einkemmer2025}. The procedure is as~follows.

\begin{enumerate}
\item
(K- and L-steps) Starting from $f_{r^0}(t_0)$ with rank $r^0$ and current basis functions $\bm X^0$ and $\bm V^0$ perform (modified) K- and L-steps independently to obtain new basis functions $\bm K^1$ and $\bm L^1$.
\item
(Basis augmentation) Let $r'= 2r$, $\{\tilde X^1_1,\dots,\tilde X^1_{r'} \}$ be an orthonormal basis (in $L_2(\Omega^{(x)})$) for the augmented set $[\bm X^0, \bm K^1]$, and $\{ \tilde V^1_1,\dots, \tilde V^1_{r'} \}$ an orthonormal basis (in $L_2(\Omega^{(v)};\omega)$) for $[\bm V^0, \bm L^1]$. Then we can represent $f(t_0)$ as
\[
f_r(t_0) = \omega  \sum_{\mu,\nu=1}^{r'} \tilde S^0_{\mu\nu} \tilde X^1_\mu \otimes \tilde V^1_{\nu}.
\]
with a new coefficient matrix $\tilde S^0 = [\tilde S^0_{ij}] \in \mathbb R^{r' \times r'}$. In this way the DLRA representation rank of $f_r$ is increased from $r$ to $r' = 2r$. With a suitable choice of basis, we can also achieve $\tilde V^1_a = U_a$ for $a = 1,\dots,m$ and a block triangular structure like~\eqref{eq:triangular S}--\eqref{eq:QR representation}.
\item
(S-step) Perform an S-step for the matrix $\tilde S$. However, instead of the ``affine'' modified S-step from~\eqref{eq:Spartial}, where only the part $\widehat S$ is evolved, here we need a ``full'' S-step as in the usual unconventional integrator. Also note that in contrast to the projector splitting integrator, the S-step in the unconventional integrator is performed ``forward'', that is, using the formulation as in~\eqref{eq:weakS} with ``$+$'' and test functions of the form
\[
w = \omega  \sum_{\mu,\nu=1}^{r'} \Sigma_{\mu \nu}  \tilde X^1_\mu \otimes\tilde V^1_{\nu}
\]
with an arbitrary matrix $\Sigma \in \mathbb R^{r' \times r'}$.
\end{enumerate}

The idea of basis augmentation is actually quite general. Instead of $\bm K^1$ and $\bm L^1$ one could use other or even more functions for basis enrichment. The above procedure would increase the rank in every time step and needs to be combined with rank truncation. Following~\cite{einkemmer2023} this is achieved as follows.

\begin{enumerate}\setcounter{enumi}{3}
\item (Rank truncation)
Assume the new solution has been computed and is of the form
\begin{align*}
f_{r'}(t_1)
&= \omega \sum_a \left( \sum_{\mu} \tilde X^1_{\mu} \tilde S^1_{\mu a} \right) \otimes U_a + \omega \sum_{\pi = m+1}^{r'} \left(\sum_{\mu} \tilde X^1_{\mu} \tilde S^1_{\mu \pi} \right) \otimes \tilde V^1_{\pi} \\
&= \omega \sum_a \tilde K_a^1 \otimes U_a + \omega \sum_{\pi = m+1}^{r'} \tilde K_{\pi}^1 \otimes \tilde W_{\pi}^{1},
\end{align*}
where the $(U_a, \tilde W_{\pi}^1)_{v,\omega} = 0$, and $r' \ge m$ is the current representation rank. Since the first part contains the fixed functions $U_a$, the rank truncation is performed on the second part only by computing a low-rank approximation
\begin{equation}\label{eq: truncation}
\sum_{\pi = m+1}^{r'} \tilde K_{\pi}^1 \otimes \tilde W_{\pi}^{1} \approx \sum_{p = m+1}^{r^1} K_{p}^1 \otimes W_{p}^{1}
\end{equation}
with $r^1 \le r'$ using singular value decomposition (SVD). Since this will imply $W_{p}^1 \in \spann \{ \tilde W_{m+1}^1, \dots, \tilde W_{r'}^{1} \} $ for all $m+1 \le p \le r^1$, the orthogonality relations $(U_a,W_{p}^1)_{v,\omega} = 0$ are preserved and one proceeds with
\[
f_{r^1}(t_1) \leftarrow \omega \sum_a K_a^1 \otimes U_a + \omega \sum_{p = m+1}^{r^1} K_{p}^1 \otimes W_{p}^{1}
\]
(with $K_a^1 = \tilde K_a^1$). Using QR decomposition of $K_1^1 \dots, K_r^1$, the result can be easily transformed back into one of the other formats such as~\eqref{eq:weighted dlra}.
\end{enumerate}

Note that the low-rank approximation~\eqref{eq: truncation} can either be based on a fixed target rank or a fixed accuracy, which results in different schemes.

\section{Adaptive discontinuous Galerkin discretization}\label{sec: DG discretization}

We now outline our approach to the solution of the Friedrichs' systems~\eqref{eq:Kvector} and~\eqref{eq:Lhat} for the K- and L-step via a DG discretization. We start by presenting the DG method for a general class of Friedrichs' systems, before specifying the discretization steps needed in our particular setting. We then analyze the conservation properties of the scheme and introduce a strategy for mesh adaptivity.

\subsection{Discontinuous Galerkin discretization for Friedrichs' systems}\label{sec: Discontinuous Galerkin discretization for Friedrichs' systems}

We first recall the discretization of Friedrichs' systems using the DG approach. Our formulation closely follows~\cite[Chapter~8]{dolejsi2015}. Other references include~\cite{hesthaven2008} and~\cite{guo2016}. In this work, we restrict ourselves to the case of periodic domains or to problems without boundary (i.e.~on $\mathbb R^d$). In the latter case, however, we assume that the \emph{computational domain} can be taken sufficiently large to fully include the support of the solution in its interior (up to machine precision) and hence can also be treated as a periodic domain numerically. We note that in principle, DG discretizations also work for domains with boundary, but this would require some more careful considerations, which will be left for future work.

Consider a Friedrichs' system
\[
\partial_t \bm  U(t)
	+ \sum_{s} \big[\mathsf A^{(s)} \, \partial_{y_s} \bm U(t) + c_s(\bm x) \mathsf B^{(s)} \, \bm U(t) \big]
		= \bm 0
\]
on an unbounded or periodic domain $\Omega^{(y)} \subset \mathbb R^d$, where $\mathsf A^{(s)} \in \mathbb R^{r \times r}$ is symmetric, $\mathsf B^{(s)} \in \mathbb R^{r\times r}$ and $c_s: \Omega^{(y)} \rightarrow \mathbb R$. Let $\bm U(t) : \Omega^{(y)} \to \mathbb R^r$ be a (continuous) exact solution. Note that the systems~\eqref{eq:Kvector} and~\eqref{eq:Lhat} for the K- and L-step are both of this form and the following considerations will be applied to both of them, essentially by replacing $y$ with $x$ or $v$.

Let $\mathcal T_{h}^{(y)}$ be a (not necessarily conforming) periodic triangulation of $\Omega^{(y)}$ and define the DG discretization space
\[
\calV_{h,p}^{(y)} = \{ U \in L_2(\Omega^{(y)}) \colon U|_T \in P_p(T), \, T \in \mathcal T_{h}^{(y)}\},
\]
where $P_p(T)$ is the space of all polynomials of degree less than or equal to $p$ on $T$. Testing the exact solution $\bm U(t)$ with $\bm \Psi \in (\calV_{h,p}^{(y)})^r$ and integrating by parts gives
\begin{equation}\label{eq: DG formulation}
\begin{split}
&\sum_{T\in \mathcal T_{h}^{(y)}} \Bigg[ \int_T  \partial_t \bm U(t) \cdot \bm \Psi \, \mathrm d \bm y
	- \int_T \sum_s \mathsf A^{(s)} \bm U(t) \cdot \partial_{y_s} \bm \Psi \, \mathrm d \bm y \\
& \quad + \int_{\partial T} \sum_s (n^{(T)}_s \, \mathsf A^{(s)} \cdot \bm U(t)) \cdot \bm \Psi \, \mathrm d S
	+  \int_T \sum_s (c_s(\bm y) \, \mathsf B^{(s)} \bm U(t)) \cdot \bm \Psi \, \mathrm d \bm y \Bigg]= 0.
\end{split}
\end{equation}
Here $\bm n^{(T)} = (n^{(T)}_s)_s$ is the outer normal vector of element $T$. Special attention in the DG discretization is given to the numerical computation of the boundary integrals over $\partial T$, as will be shown next.

Let $\mathcal F_h^{(y)}$ be the set of all $(d-1)$-dimensional faces of elements in $\mathcal T_{h}^{(y)}$. Since we assume that the domain is periodic, all faces are in fact internal faces. For each $e \in \mathcal F_h^{(y)}$ we define a fixed unit normal vector $\bm n^{(e)}$ (the direction can be arbitrary). Taking into account that the domain is periodic, every boundary $\partial T$ of an element essentially decomposes into shared faces with its neighboring elements. For every face $e\in \mathcal F_h^{(y)}$, there exist two elements $T^+$ and $T^-$ such that $e \subset T^+ \cap T^-$. We use the convention that $T^+$ lies in the direction of $\bm n^{(e)}$. Given a vector field $\bm U$ on $\Omega^{(y)}$, one then defines the average and jump of $\bm U$,
\[
\{\bm U\}_e = \frac 1 2 (\bm U_e^+ + \bm U_e^-), \qquad [\bm U]_e = \bm U_e^- - \bm U_e^+,
\]
across the face $e \in \mathcal F_h^{(y)}$, where
\[ 
\bm U_e^\pm(\bm y) = \lim_{\epsilon \rightarrow 0^+} \bm U(\bm y \pm \epsilon \bm n^{(e)}).
\]

Now, again taking periodicity into account, every face $e$ with neighboring elements $T^+$ and $T^-$ appears in precisely two boundary integrals in~\eqref{eq: DG formulation} (for $T^+$ and $T^{-}$). Using $\bm n^{(T^\pm)} = \mp \bm n^{(e)}$, the overall contribution of $e$ in the boundary integrals is therefore
\[
\int_e \sum_s (n^{(T^+)}_s\, \mathsf A^{(s)} \bm U) \cdot \bm \Psi^+_e + (n^{(T^-)}_s \, \mathsf A^{(s)} \bm U) \cdot \bm \Psi^-_e \, \mathrm d S
=\int_e \bm F_e(\bm U) \cdot [\bm \Psi]_e\, \mathrm d S,
\]
where
\[
\bm F_e(\bm U)= \mathsf A_e \cdot \bm U  \quad \text{with} \quad  \mathsf A_e \coloneqq \sum_s n^{(e)}_s \, \mathsf A^{(s)}
\] 
is the so-called flux over the edge $e$. 

For a discontinuous $\bm U_h \in \big(\calV_{h,p}^{(y)}\big)^r$, the flux has to be approximated by a \emph{numerical flux}, that is,
\[
\bm F_e(\bm U_h) \approx \bm F_e^*\big( (\bm U_h)_e^-, (\bm U_h)_e^+).
\]
Common choices are based on the spectral decomposition $\mathsf A_e = Q \Lambda Q^{-1}$, where $\Lambda$ is diagonal. We define its positive and negative part as $\mathsf A_e^\pm = Q \Lambda^\pm Q^{-1}$, where $\Lambda^\pm = \pm \max(\pm \Lambda,0)$. For $0 \le \alpha \le 1 $ we introduce the numerical flux
\begin{equation} \label{eq:flux}
\bm F_e^*\big( (\bm U_h)^-_e, (\bm U_h)^+_e\big) =  \mathsf A_e \{\bm U_h\}_e + \frac{1-\alpha}{2} \,  \lvert \mathsf A_e \rvert \cdot [\bm U_h]_e
\end{equation}
where $\lvert \mathsf A_e \rvert = \mathsf A_e^+ - \mathsf A_e^-$. Choosing $\alpha=0$ recovers the so-called \emph{upwind flux}, while $\alpha=1$ corresponds to the \emph{central flux}.

The resulting semi-discrete Friedrichs' system can finally be written as
\begin{equation*}
\begin{split}
&\sum_{T\in \mathcal T_{h}^{(y)}} \int_{T} \partial_t \bm U_h(t) \cdot \bm \Psi
	+ \sum_s \Big[
	-(\mathsf A^{(s)} \bm U_h(t)) \cdot \partial_{y_s} \bm \Psi_h
		+ (c_s(\bm y) \, \mathsf B^{(s)} \bm U_h(t)) \cdot \bm \Psi_h
		\Big]
	\, \mathrm d \bm y \\
&+ \sum_{e\in\mathcal F_h^{(y)}} \int_e \bm F_e^*\big((\bm U_h(t))_e^-, (\bm U_h(t))_e^+\big) \cdot [\bm \Psi_h]_e \, \mathrm d S
= 0 \quad \text{for all } \bm \Psi_h \in \big(\calV^{(y)}_{h,p}\big)^r.
\end{split}
\end{equation*}

For the time integration, we employ an explicit Euler step with step size $\tau$ for computing a new discrete solution $\bm U_h^{n+1}$ from a given one $\bm U_h^n$, leading to the fully discrete system
\begin{equation} \label{eq:DGfully}
\begin{split}
&\sum_{T\in \mathcal T_{h}^{(y)}} \int_{T} \frac 1 \tau \big(\bm U_h^{n+1} - \bm U_h^n\big) \cdot \bm \Psi
	+ \sum_s \Big[
	-(\mathsf A^{(s)} \bm U_h^n) \cdot \partial_{y_s} \bm \Psi_h
		+ (c_s(\bm y) \, \mathsf B^{(s)} \bm U_h^n) \cdot \bm \Psi_h
		\Big]
	\, \mathrm d \bm y \\
&+ \sum_{e\in\mathcal F_h^{(y)}} \int_e \bm F_e^*\big((\bm U_h^n)_e^-, (\bm U_h^n)_e^+\big) \cdot [\bm \Psi_h]_e \, \mathrm d S
= 0 \quad \text{for all } \bm \Psi_h \in \big(\calV^{(y)}_{h,p}\big)^r.
\end{split}
\end{equation}

Before proceeding, we need to discuss an important aspect regarding the particular system matrices in the Friedrichs' systems~\eqref{eq:Kvector} and~\eqref{eq:Lhat} of the K- and L-step. By~\eqref{eq:Axs_Bxs} and~\eqref{eq:Avs_Bvs}, they arise as Galerkin-type matrices involving derivatives such as
\[
\big(\partial_{v_s}(\omega V_j), V_i \big)_v, \quad
\big(\partial_{x_s} X_j, X_i \big)_x.
\]
However, in the DG setting, these terms are not defined since we are differentiating possibly discontinuous functions (e.g.~$\partial_{x_s} X_j)$ and integrate with another discontinuous function (e.g.~$X_i$). Hence, a reasonable approximation is needed which we will do as follows. For $U_h, V_h \in \calV_{h,p}^{(y)}$ and $s = 1,\dots,d$, we define discrete partial derivatives as linear forms
\begin{equation} \label{eq:dfg}
\big( \hat{\mathrm d}_{y_s} U_h, V_h\big)_y 
\coloneqq \sum_{T\in \mathcal T_{h}^{(y)}}
	 \int_T \partial_{y_s} U_h \, V_h \, \mathrm d \bm y
		- \sum_{e \in \mathcal F_h^{(y)}} \int_e  n_s^{(e)} \, [U_h]_e \, \{ V_h \}_e \, \mathrm d S.
\end{equation}
This definition is a generalization of the distributional derivative of a function $U_h \in \calV_{h,p}^{(y)}$ tested with a smooth and compactly supported function $\varphi \in C_0^\infty(\Omega^{(y)})$:
\[
\int_{\Omega^{(y)}} \partial_{y_s} U_h \, \varphi \, \mathrm d\bm y =
	\sum_{T \in \mathcal T_{h}^{(y)}} \int_T \partial_{y_s} U_h \, \varphi \, \mathrm d \bm y
		- \sum_{e\in \mathcal F_h^{(y)}} \int_e n_s^{(e)} [U_h]_e \, \varphi \, \mathrm d S.
\]
Now, if $\varphi$ itself is discontinuous along the edges $e$, the evaluation of $\varphi$ at an edge is not well-defined. Hence for $\varphi \in \calV_{h,p}^{(y)}$ we take the average value $\{\varphi\}_e$ on the edge, which gives~\eqref{eq:dfg}.

The discrete partial derivatives satisfy a discrete version of an integration by parts formula, which will be important for proving conservation properties in section~\ref{sec: Conservativion of physical invariants}.

\begin{lemma} \label{lemma:dhat}
Let $U_h,V_h \in \calV_{h,p}^{(y)}$ on a periodic domain. Then
\begin{equation*}
\big( \hat{\mathrm d}_{y_s} U_h, V_h \big)_y
= - \sum_{T \in \mathcal T_{h}^{(y)}} \int_T U_h \partial_{y_s} V_h \, \mathrm d \bm y
	+ \sum_{e \in \mathcal F_h^{(y)}} \int_e n_s^{(e)} \{U_h\}_e \, [V_h]_e \, \mathrm d S.
\end{equation*}
\end{lemma}
\begin{proof}
Inserting the definition~\eqref{eq:dfg} and applying integration by parts on each element $T$ yields
\[
\big( \hat{\mathrm d}_{y_s} U_h, V_h \big)_y
= - \sum_{T\in \mathcal T_{h}^{(y)}} \int_T U_h \partial_{y_s} V_h \, \mathrm d \bm y
	+ \sum_{e \in \mathcal F_h^{(y)}} \int_e n_s^{(e)} [U_h \, V_h]_e \, \mathrm d S 
	- \sum_{e \in \mathcal F_h^{(y)}} \int_e  n_s^{(e)} \, [U_h]_e \, \{ V_h \}_e \, \mathrm d S
\]
due to the periodicity of the triangulation. Noting that
\begin{align*}
[U_h \, V_h]_e - [U_h]_e \, \{V_h\}_e
&= (U_h)_e^- \, (V_h)_e^- - (U_h)_e^+ \, (V_h)_e^+ - \frac 1 2 \big((U_h)_e^- - (U_h)_e^+ \big) \cdot \big((V_h)_e^- + (V_h)_e^+\big) \\
&= \frac 1 2\big( (U_h)_e^- \, (V_h)_e^- - (U_h)_e^- \, (V_h)_e^+ + (U_h)_e^+ \, (V_h)_e^- - (U_h)_e^+ \, (V_h)_e^+ \big) \\
&= \{U_h\}_e \, [V_h]_e
\end{align*}
completes the proof.
\end{proof}

\subsection{Application of DG discretization to DLRA}\label{sec: Application of DG discretization to DLRA}

We now apply the DG method to discretize the equations arising in the K- and L-steps~\eqref{eq:Kvector} and \eqref{eq:Lhat}. We assume the basis functions $X_i$ and $V_j$ are represented in possibly different DG spaces $\calV_{h,p}^{(x)}$ and $\calV_{h,q}^{(v)}$. For notational simplicity we drop the index $h$ for all discretized functions in the following.

The procedure for one time step consists of the following parts: computing the electric field, assembling the system matrices in the Friedrichs' formulation, and solving them. Finally, following the unconventional integrator, the resulting basis functions are augmented, and an S-step is performed, followed by possible rank truncation as explained in section~\ref{sec: Modified unconventional integrator}. We do not present the details for the fixed-rank projector splitting version as they are very similar and even simpler because no basis augmentation and truncation are required.

The electric field is calculated via equation~\eqref{eq:poisson}. The charge density generated by $f_r$ is given by
\[
\rho(t, \bm x) = \sum_{ij} X_i(\bm x) S_{ij} \int_{\Omega^{(v)}} V_j(\bm v) \, \mathrm d \bm v
\]
and is in $\calV_{h,p}^{(x)}$. Hence, the Poisson equation can be solved exactly on the same mesh with continuous functions of degree $p+2$.

We first focus on the K-step~\eqref{eq:Kvector} using the explicit Euler method~\eqref{eq:DGfully}. We assume that the DLRA representation is in the block QR format~\eqref{eq:triangular S}-\eqref{eq:QR representation}. As discussed at the end of the previous subsection, the system matrices $\mathsf B^{(x,s)}$ are not well-defined and need to be approximated. In particular, we define
\begin{equation*} \label{eq:Axs_Bxs_2}
\mathsf A^{(x,s)} = \Big[ \big( v_s V_j^n, V_i^n \big)_{v,\omega} \Big]_{ij}, \qquad
	\mathsf B_h^{(x,s)} = \Big[  \big(\hat{\mathrm d}_{v_s}(\omega V_j^n), V_i^n \big)_{v} \Big]_{ij},
\end{equation*}
using the discrete partial derivatives $\hat{\mathrm d}_{y_s}$, defined in \eqref{eq:dfg}, instead of $\partial_{y_s}$ in $\mathsf B^{(x,s)}$. Hence, the discretized equations for the K-step read
\begin{equation} \label{eq:Kupdate}
\begin{split}
\frac 1 \tau ( \bm K^{n+1} - \bm K^n, \bm \psi)_x 
& + \sum_{T \in \mathcal T_{h}^{(x)}}
		\sum_s \big[
			\int_T -(\mathsf A^{(x,s)} \cdot \bm K^{n}) \cdot \partial_s \bm \psi 
				- E^n_s (\mathsf B_h^{(x,s)} \cdot \bm K^n) \cdot \bm \psi \, \mathrm d \bm x
	\big] \\
&+ \sum_{e \in \mathcal F_h^{(x)}} \int_e \bm F^*_e(\bm K_e^-, \bm K_e^+) \cdot [\bm \psi]_e \, \mathrm d S 
= 0
\quad \text{for all } \bm \psi \in (\calV_{h,p}^{(x)})^r.
\end{split}
\end{equation}

The Friedrichs' system~\eqref{eq:Lhat} for the velocity components is formulated for weighted functions of the form $\widehat{\bm L}_\omega = \omega \bm L$, where the individual components $L_p$ belong to $L_2(\mathbb R^d; \omega)$. This needs to be reflected in the discrete equations. Keeping the notation from the previous subsection, assume we are given a discrete function $L \in \calV_{h,p}^{(v)}$ (representing a component  $L_p$). We then need a discrete representation $\widehat L_\omega \in \calV_{h,p}^{(v)}$ of the $L_2$ function $\omega L$. For this, we compute the $L_2$ projection of $\omega L$ onto $\calV_{h,p}^{(v)}$, which defines a linear operator
\[ 
\mathcal P_{\omega} (L) = \widehat L_\omega \quad \textnormal{such that} \quad
(\widehat L_\omega, \psi)_v = (\omega  L, \psi)_v \quad \text{for all $\psi \in \calV_{h,p}^{(v)}$.}
\]
On $\calV_{h,p}^{(v)}$, the operator $\mathcal P_{\omega}$ is invertible. The inverse operator is hence used as a substitute for the multiplication by $1/\omega$.

As a result, for computing the L-step we first project
\[
\widehat L_{\omega,p}^n = \mathcal P_\omega \left( \sum_q S^n_{pq} W^n_q \right), \qquad \widehat{\bm L}_\omega^n = [\widehat L^n_{\omega,p}]_p,
\]
and then perform an explicit Euler step of the DG discretization
\[
\begin{split}
\frac 1 \tau (\widehat{\bm L}_\omega^{n+1} - \widehat{\bm L}_\omega^{n}, \bm \psi)_v 
& + \sum_{T \in \mathcal T_{h}^{(v)}}
		\sum_s 
			\int_T \big(-\widehat{\mathsf A}^{(v,s)} \cdot \widehat{\bm L}^{n}_\omega\big) \cdot \partial_s \bm \psi 
				+ v_s (\widehat{\mathsf B}_h^{(v,s)} \cdot \widehat{\bm L}^n) \cdot \bm \psi \, \mathrm d \bm v
	 \\
&+ \sum_{e \in \mathcal F_h^{(v)}} \int_e \bm F^*_e\big((\widehat{\bm L}_\omega^n)_e^-, (\widehat{\bm L}_\omega^n)_e^+\big) \cdot [\bm \psi]_e \, \mathrm d S \\
&= -\sum_s
		\big( -\bar{\mathsf A}^{(v,s)} \bar S \partial_{v_s}( \omega \bm U)
			+  v_s  \bar{\mathsf B}_h^{(v,s)} \bar S \omega \bm U, \, \widehat{\bm \psi}^{(v)}\big)_v, 
\end{split}
\]
where
\[
\mathsf A^{(v,s)} = \Big[ \big(-E_s X_j^n, X_i^n \big)_x \Big]_{ij}, \qquad
	\mathsf B_h^{(v,s)} = \Big[  \big( \hat{\mathrm d}_{x_s}  X_j^n, X_i^n \big)_{x} \Big]_{ij}.
\]
and its parts are defined as in~\eqref{eq:hatbarv}. The L-step is completed by applying the inverse of the projector to calculate
\[ 
L_p^{n+1} = \mathcal P_\omega^{-1}\big(\widehat{L}_{\omega, p}^{n+1}\big).
\]

Finally, for the S-step we use the unconventional integrator with an augmented basis. We define orthonormal bases $\tilde{\bm X}^{n+1} \in (\calV_{h,p}^{(x)})^{2r}$ and $\tilde{\bm V}^{n+1} \in (\calV_{h,p}^{(v)})^{2r - m}$ for the augmented matrices $[ \bm X^n, \bm K^{n+1}]$ and $[\bm V^n, \bm L^{n+1}]$, respectively. Here, $\bm L^{n+1}$ contains the new basis functions $L^{n+1}_p$ from the L-step. The orthonormal basis $\tilde{\bm V}^{n+1}$ is always chosen with respect to the $(\cdot,\cdot)_{v,\omega}$ inner product and such that the first $m$ functions remain the fixed $U_a$.

For expressing functions with respect to these new bases we let
\[ 
M_{ki} = (\tilde X^{n+1}_k, X^n_i)_x, \quad N_{\ell j} = (\tilde V^{n+1}_\ell, V^n_j)_{v,\omega}
\]
with $k$ running from $1,\dots,2r$ and $\ell$ from $1,\dots,2r-m$. It follows that the state at the previous ($n$-th) step can be expressed in the new basis as
\begin{equation} \label{eq:Sn}
\sum_{ij} S_{ij}^n X_i^n \otimes V_j^n = \sum_{k\ell} \tilde S^{n}_{k\ell} \tilde X_k^{n+1} \otimes \tilde V_\ell^{n+1}, \quad
\tilde S^{n}_{k\ell} = \sum_{ij} M_{ki} S_{ij}^{n} N_{\ell j}.
\end{equation}
To update $S$ for the function
\[
f(t) = \omega \sum_{k\ell} \tilde S_{k\ell}(t) \tilde X^{n+1}_k \otimes \tilde V^{n+1}_\ell,
\]
we consider the ODE as in~\eqref{eq:weakS} but on the subspace $T_{\tilde{\bm X}^{n+1}, \tilde{\bm V}^{n+1}}$ of the augmented bases. Applying the explicit Euler method for the time discretization and using~\eqref{eq:Sn} then gives
\begin{equation} \label{eq:Skl_step}
\begin{split}
\tilde S_{k\ell}^{n+1} 
=  \tilde S_{k\ell}^{n}
- \tau \sum_{ijs}
	\Big[ 
&		\big( \hat{\mathrm d}_{x_s} X_i^n, \tilde X_k^{n+1} \big)_{x} \, S_{ij}^n  \, (\tilde V_\ell^{n+1}, v_s V_j^n)_{v,\omega} \\
&		+ (\tilde X_k^{n+1}, - E_s X_i^n)_x  \,  S_{ij}^n  \, \big( \hat{\mathrm d}_{v_s} (\omega V_j^n), \tilde V_\ell^{n+1} \big)_{v}
	\Big],
\end{split}
\end{equation}
where as in the K- and L-step above the system matrices employ the discrete derivatives. Once a new $S^{n+1}$ has been computed, the rank truncation procedure is applied as explained in section~\ref{sec: Modified unconventional integrator}.

\subsection{Conservation of mass and momentum}\label{sec: Conservativion of physical invariants}

In this section, we establish the mass and momentum conservation properties of the unconventional integrator applied to the DG discretization when using central fluxes. Specifically, following~\cite{einkemmer2025,einkemmer2023}, we prove that the numerical scheme satisfies discrete continuity relations consistent with the analytic continuity equations. As a weight function $\omega(\bm v)$, we use the Gaussian weight~\eqref{eq: Gauss weight}, although generalizations are possible.

For a density function $f$ the mass density (or charge density) and the momentum density (or current density) are defined as
\[
\rho(t,\bm x)
	= \int_{\Omega^{(v)}} f(t, \bm x, \bm v) \, \mathrm d \bm v, \qquad
\bm j(t,\bm x)
	= \int_{\Omega^{(v)}} \bm v \, f(t,\bm x, \bm v) \, \mathrm d \bm v.
\]
For the continuous solution of the Vlasov--Poisson equation~\eqref{eq:transport}, these functions satisfy the continuity equations
\begin{equation}\label{eq: continuity rho}
\partial_t \rho(t) + \nabla_x \cdot \bm j(t) = 0
\end{equation}
and, for $s=1,\dots,d$,
\begin{equation}\label{eq: continuity j}
\partial_t j_s(t) + \nabla_x \cdot \bm \sigma_s(t) = -E_s(t) \rho(t), \qquad \bm \sigma_s (t,\bm x) = \Big[ \int_{\Omega^{(v)}} v_s v_{s'} \, f(t,\bm x, \bm v) \, \mathrm d \bm v \Big]_{s'=1,\dots,d}.
\end{equation}
These equations, in particular, imply mass and momentum conservation over time. 

In the numerical scheme from section~\ref{sec: Application of DG discretization to DLRA}, instead of $\rho$ and $j_s$, discrete quantities $\rho^n$ and $j^n_s$ are computed for $n=0,1,2\dots,$ from the discrete bases $\bm X^n, \bm V^n$ and coefficient matrices $S^n$. The goal is to establish local continuity relations for the discrete quantities. For this we adapt the approach recently presented in~\cite{einkemmer2025}.

\begin{lemma} \label{lemma:fhat}
For $f^n = \omega \sum_{ij} S^n_{ij} X^n_i \otimes  V^n_j$ define
\begin{equation*} \label{eq:Fhat}
	\hat F(f^n) := \sum_{s} -v_s \, \hat{\mathrm d}_{x_s} f^n + E^n_s \, \hat{\mathrm d}_{v_s} f^n.
\end{equation*}
Then, using the central flux, i.e.,~$\alpha = 1$ in \eqref{eq:flux}, the K-step \eqref{eq:Kupdate} can be written as
\begin{equation} \label{eq:Kupdate_Fhat}
	(K^{n+1}_j, \psi)_x = (K^n_j, \psi)_x + \tau \big(\hat F(f^n), \psi \otimes V^n_j\big)_{xv}, \quad
	j=1,\ldots,r,
\end{equation}
for all $\psi \in \mathcal V^{(x)}_{h,p}$. We also have
\begin{equation} \label{eq:Supdate_Fhat}
	\tilde S^{n+1}_{k\ell} = \tilde S^{n}_{k\ell} 
		+ \tau \big( \hat F(f^n), \tilde X^{n+1}_k \otimes \tilde V^{n+1}_\ell \big)_{xv} 
\end{equation}
for the S-step \eqref{eq:Skl_step}.
\end{lemma}
\begin{proof}
For the first equation, recall that $f^n = \omega \sum_j K^n_j \otimes V^n_j$, which gives
\begin{align*}
	\big(\hat F(f^n), \psi \otimes V^n_k\big)_{xv}
	&= \sum_s \sum_j 
		\big( -v_s \omega V_j^n, V^n_k \big)_v \cdot 
			\big( \hat{\mathrm d}_{x_s} K^n_j, \psi\big)_x
		+ \big(  \hat{\mathrm d}_{v_s}(\omega V_j^n), V^n_k \big)_v \cdot 
			\big( E^n_s K^n_j, \psi\big)_x \\
	&= \sum_s \sum_j -\mathsf A^{(x,s)}_{kj} \, \big( \hat{\mathrm d}_{x_s} K^n_j, \psi\big)_x
		- \mathsf B^{(x,s)}_{h,kj} \, 	\big( E^n_s K^n_j, \psi\big)_x\, . 
\end{align*}
Now, Lemma \ref{lemma:dhat} implies
\[
\big( \hat{\mathrm d}_{x_s} K^n_j, \psi \big)_x
= - \sum_{T \in \mathcal T_{h}^{(x)}} \int_T K^n_j \partial_{x_s} \psi \, \mathrm d \bm x
	+ \sum_{e \in \mathcal F_h^{(x)}} \int_e n_s^{(e)} \{K^n_j\}_e \, [\psi]_e \, \mathrm d S.
\]
Plugging this into the previous equation yields shows that~\eqref{eq:Kupdate_Fhat} is just~\eqref{eq:Kupdate} in component form for central fluxes with $\alpha = 1$.

The second equation is immediate from~\eqref{eq:Skl_step} and the representation of $f^n$.
\end{proof}

Next note that both the discrete density and momentum density can be written in the form \(\varphi_U^n = \big( f^n, U \big)_v\) for appropriate functions $U$: for the density $\rho$ we take $U = 1$ whereas for $j_s$ we take $U = v_s$. By construction, all these functions are elements of the subspace $\spann(U_a)$ spanned by the fixed velocity functions, see section~\ref{sec: Modified DLRA model}. We now derive a discrete equation of motion for $\varphi_U^n$  which eventually will eventually lead to discrete versions of the continuity equations~\eqref{eq: continuity rho} and~\eqref{eq: continuity j}.

\begin{lemma} \label{lemma:phiU}
Let $U \in \spann(U_a)$ and define
\[
	\varphi_U^n = \big( f^n, U \big)_v \in \mathcal V^{(x)}_{h,p}.
\]
Then, these quantities satisfy the update equation
\begin{equation}\label{eq: update equation phi}
	\big(\varphi_U^{n+1},\psi\big)_x 
		= \big(\varphi_U^n, \psi\big)_x 
			+ \tau \Big( \big(\hat F(f^n), U \big)_{v}, \psi\Big)_x \quad
			\text{for all } \psi \in V^{(x)}_{h,p}
\end{equation}
when the algorithm described in section~\ref{sec: Application of DG discretization to DLRA} is applied with central fluxes.
\end{lemma}

\begin{proof}
Let $P$ denote the orthogonal projection of $\big(\hat F(f^n), U)_v$ onto $\mathcal V^{(x)}_{h,p}$. We first show that $P \in \spann(\tilde{\bm X}^{n+1})$. Recall that
\begin{equation}\label{eq:U_projection}
U = \sum_j \big(V^{n}_j, U\big)_{v,\omega}  V^{n}_j
\end{equation}
since, by construction, $U$ belongs to $\spann({\bm V}^{n})$ and in the algorithm it is always ensured that the $V^n_j$ form an orthonormal system in the weighted $(\cdot, \cdot )_{v, \omega}$ inner product. Therefore, it holds that
\begin{align*}
	\Big( \big(\hat F(f^n), U)_v, \psi \Big)_x
	&= \Big( \sum_j \big(\hat F(f^n), V^n_j\big)_v \, \big(V^n_j, U \big)_{v,\omega},  \psi \Big)_x
		= \sum_j \big( \hat F(f^n), \psi \otimes V^n_j\big)_{xv} \, \big(V^n_j, U \big)_{v,\omega} \\
	&= \Big( \sum_j \frac 1 \tau \big[K^{n+1}_j - K^n_j \big]\, \big(V^n_j, U \big)_{v,\omega}, \psi \Big)_x
\end{align*}
for all $\psi \in \mathcal V^{(x)}_{h,p}$, where we used~\eqref{eq:Kupdate_Fhat}. This shows
\[
 P = \sum_j \frac 1 \tau \big[K^{n+1}_j - K^n_j \big]\, \big(V^n_j, U \big)_{v,\omega}
		\in \spann(\tilde{\bm X}^{n+1})
\]
since $K^{n}_j, K^{n+1}_j \in \spann(\tilde{\bm X}^{n+1})$. As a result, since the $\tilde X_i^{n+1}$ are an orthonormal system by construction, we can write
\[
P = \sum_i \Big( \tilde X^{n+1}_i, \big(\hat F(f^n),U\big)_v\Big)_x \, \tilde X^{n+1}_i
\]
and hence it holds that
\begin{equation}\label{eq:Fhat_projection}
 \Big( \big(\hat F(f^n), U)_v, \psi \Big)_x =
		\Big( \sum_i \Big( \tilde X^{n+1}_i, \big(\hat F(f^n),U\big)_v\Big)_x \, \tilde X^{n+1}_i, \psi \Big)_x
\end{equation}
for all $\psi \in \mathcal V^{(x)}_{h,p}$. We can now prove~\eqref{eq: update equation phi}. Using~\eqref{eq:Supdate_Fhat}, we have
\begin{align*}
	\big( \varphi^{n+1}_U, \psi\big)_x
	& = \Big(
	\big( \omega \sum_{ij} \tilde S^{n+1}_{ij} \, \tilde X^{n+1}_i \otimes \tilde V^{n+1}_j, \, U \big)_v,
	\psi \Big)_x \\
	& = \Big( \sum_{ij}
		\Big( \omega
			\Big[ \tilde S^{n}_{ij}
				+ \tau \big( \hat F(f^n), \tilde X^{n+1}_i \otimes \tilde V^{n+1}_j \big)_{xv}
			\Big]
			\tilde X^{n+1}_i \otimes \tilde V^{n+1}_j, \, U \Big)_v, \psi\Big)_x \\
	& = \big( \varphi^{n}_U, \psi\big)_x
		+ \tau \Big( \sum_{ij} \Big(\omega \big( \hat F(f^n), \tilde X^{n+1}_i \otimes \tilde V^{n+1}_j \big)_{xv}
			\tilde X^{n+1}_i \otimes \tilde V^{n+1}_j, \, U \Big)_v, \psi\Big)_x.
\end{align*}
Due to~\eqref{eq:U_projection} the last expression equals
\[
 \big( \varphi^{n+1}_U, \psi\big)_x = \big( \varphi^{n}_U, \psi\big)_x
		+ \tau \Big( \sum_i \tilde X^{n+1}_i \big(\hat F(f^n), \tilde X^{n+1}_i \otimes U \big)_{xv}, \psi\Big)_x,
\]
which in light of~\eqref{eq:Fhat_projection} shows~\eqref{eq: update equation phi}. Finally, this relation is preserved if the conservative rank truncation~\eqref{eq: truncation} from~\cite{einkemmer2023} is applied.
\end{proof}

Based on Lemma~\ref{lemma:phiU} we now establish the discrete continuity equations.

\begin{theorem}
If $m \ge 1$, the computed discrete mass densities $\rho^n$ fulfill (in exact arithmetic) for all $n=0,1,2,\dots$ the relations
\begin{equation} \label{eq:dg_rho}
\frac 1 \tau (\rho^{n+1} - \rho^n, \psi)_x
 -\sum_{T \in \mathcal T_{h}^{(x)}} \int_{T} \bm j^n \cdot \nabla \psi \, \mathrm d x
 + \sum_{e\in \mathcal F_h^{(x)}} \int_e \big(\bm n^{(e)} \cdot \{ \bm j^n\}_e \big)  \, [\psi]_e \, \mathrm d S
 = 0
\end{equation}
for all $\psi \in \calV_{h,p}^{(x)}$. This corresponds to a DG discretization of the continuity equation~\eqref{eq: continuity rho} using central fluxes and explicit Euler time stepping as in~\eqref{eq:DGfully}. Hence, the discrete local continuity equation holds for the mass density, and the total mass $\int_{\Omega^{(x)}} \rho^n \, \mathrm d \bm x$ is conserved.
\end{theorem}
\begin{proof}
Letting $U=1$, so that $\rho = \varphi_U$, and applying Lemma \ref{lemma:phiU}, we obtain
\[
\rho^{n+1} = \rho^n + \tau \Big( \big(\hat F(f^n), 1 \big)_{v}, \psi\Big)_x\,,
\]
with
\begin{align*}
	\Big( \big(\hat F(f^n), 1 \big)_{v}, \psi\Big)_x
	&= \sum_s -\Big(\hat{\mathrm d}_{x_s}\big( v_s, f\big)_v, \psi\Big)_x  
		+ \Big( E_s \big(\hat{\mathrm d}_{v_s} f^n, 1)_v, \psi \Big)_x
\end{align*}
for all $\psi \in \calV_{h,p}^{(x)}$. Lemma~\ref{lemma:dhat} implies
\[ 
\big( \hat{\mathrm d}_{v_s} f^n, 1 \big)_{v} = 0.
\]
Applying the same lemma to the first term gives
\begin{align*}
	\Big( \big(\hat F(f^n), 1 \big)_{v}, \psi\Big)_x
	&= \sum_s 
		\sum_{T \in \mathcal T_{h}^{(x)}} \int_T j_s \partial_{x_s} \psi \, \mathrm d \bm x
		- \sum_{e \in \mathcal F_h^{(x)}} \int_e n_s^{(e)} \{j_s\}_e \, [\psi]_e \, \mathrm d S,
\end{align*}
thereby showing \eqref{eq:dg_rho}. Letting $\psi = 1 \in \calV_{h,p}^{(x)}$ in \eqref{eq:dg_rho} demonstrates conservation of total mass.
\end{proof}

\begin{theorem}
Assume a central flux, i.e.~$\alpha = 1$ in~\eqref{eq:flux}, for computing the K-steps~\eqref{eq:Kupdate}. If $m \ge 1 + d$, then the computed discrete mass densities $j^n_s$ fulfill (in exact arithmetics) for all $n=0,1,2,\dots$ the relations
\begin{equation} \label{eq:jupdate}
\frac 1 \tau (j_s^{n+1} - j_s^n, \psi)_x
 +\sum_{T\in \mathcal T_{h}^{(x)}} \int_T -\bm \sigma_s^n \cdot \nabla \psi + E^n_s \rho^n \psi\, \mathrm d \bm x
 + \sum_{e\in \mathcal F_h^{(x)}} \int_e \big(\bm n^{(e)} \cdot \{ \bm \sigma_s^n\}_e \big)  \, [\psi]_e \, \mathrm d S
 = 0
\end{equation}
for all $\psi \in \calV_{h,p}^{(x)}$. This corresponds to a DG discretization of the continuity equations~\eqref{eq: continuity j} using central fluxes and explicit Euler time stepping as in~\eqref{eq:DGfully}. Hence, the discrete local continuity equation holds for the momentum density and the total momentum $\int_{\Omega^{(x)}} \bm j^n \, \mathrm d \bm x$ is conserved.
\end{theorem}
\begin{proof}
Letting $U=v_s$, so that $j_s = \varphi_U$, and applying Lemma \ref{lemma:phiU}, we obtain
\[
	j_s^{n+1} = j_s^n + \tau \Big( \big(\hat F(f^n), v_s \big)_{v}, \psi\Big)_x\,,
\]
with
\begin{align*}
	\Big( \big(\hat F(f^n), v_s \big)_{v}, \psi\Big)_x
	&= \sum_{s'} -\Big(v_{s'} \, \hat{\mathrm d}_{x_{s'}}\big( v_s, f\big)_v, \psi\Big)_x  
		+ \Big( E^n_{s'} \big(\hat{\mathrm d}_{v_{s'}} f^n, v_s)_v, \psi \Big)_x
\end{align*}
for all $\psi \in \calV_{h,p}^{(x)}$. In this case Lemma~\ref{lemma:dhat} yields
\[ 
\big( \hat{\mathrm d}_{v_s'} f^n, v_s \big)_{v} = \delta_{s',s} \rho^n.
\]
Applying the same lemma also to the first term hence gives
\begin{align*}
	\Big( \big(\hat F(f^n), v_s \big)_{v}, \psi\Big)_x
	= \sum_{s} & 
		\sum_{T \in \mathcal T_{h}^{(x)}} \int_T \sum_{s'} \sigma_{s,s'} \partial_{x_s} \psi \, \mathrm d \bm x
		- \sum_{e \in \mathcal F_h^{(x)}} \int_e n_s^{(e)} 
			\big\{\sum_{s'} \sigma_{s,s'}\big\}_e \, [\psi]_e \, \mathrm d S \\
	&+ \big( E^n_s \rho^n, \psi\big)_x \,,
\end{align*}
thereby showing \eqref{eq:jupdate}. For the conservation of momentum, we note that in the continuous setting one has~\cite{einkemmer2023}
\[
\bm E (1-\rho)=\nabla \cdot \left(\bm E \otimes \bm E - \frac{1}{2} \|\bm E\|^2\right),
\]
and hence
\[
\int_{\Omega^{(x)}} \bm E (1-\rho)\, \mathrm d \bm x = 0.
\]
This also holds for the discrete version, since the computed $\bm E^n$ is the exact electric field corresponding to the (discrete) density $\rho^n$; see Section~\ref{sec: Application of DG discretization to DLRA}. Setting $\psi = 1 \in \calV_{h,p}^{(x)}$ in~\eqref{eq:jupdate}, and using the fact that $\int_{\Omega^{(x)}} \bm E \, \mathrm d \bm x = 0$, we conclude that the total momentum is conserved.
\end{proof}

We remark that the key observation for deriving the local continuity equations--and thus the associated conservation properties--is that the updates for $\bm K$ and $S$ can be expressed using $\hat F$ (Lemma~\ref{lemma:fhat}). However, this property holds only for the central flux. For the upwind flux additional jump terms appear and the local continuity equations in the presented form are hence not satisfied when using the upwind flux. Depending on the basis $\tilde{\bm X}^{n+1}$, physical quantities may still be conserved globally. It remains to be investigated whether the algorithm can be modified to also conserve these quantities when employing the upwind flux.

\subsection{Spatial adaptivity}\label{sec: spatial adaptivity}

In this section, we explore the use of adaptive mesh strategies within the DG framework for DLRA. As a demonstration of feasibility, we present a rather simple adaptive scheme and present numerical tests in section~\ref{sec: transport adaptive DG} to illustrate its behavior. The development of more sophisticated schemes and error estimators, like in~\cite{Wieners2023}, is left for future work.

We consider the general setup of section~\ref{sec: DG discretization} (which applies to both the K- and L-step) and employ an elementwise error indicator based on the projection of a DG solution. Specifically, given a DG function $\bm U\in (V_{h,p}^{(y)})^r$ of polynomial order $p$, we project it onto a DG space of order $p-1$ within each element $T \in \mathcal T_{h}^{(y)}$, resulting in a function $\tilde{\bm U} \in V_{h,p-1}^{(y)}$ that satisfies
\[ 
( \tilde{\bm U}, \varphi)_{L_2(\Omega^{(y)})} = (\bm U, \varphi)_{L_2(\Omega^{(y)})} \quad \text{for all } \varphi \in \calV_{h,p-1}^{(y)}.
\]
The local error indicator is then defined as the maximum of the component wise \( L^2 \)-norm of the difference between the original and the projected function, that is,
\[
e^{(T)} = \max_i \, \| \tilde U_i - U_i\|_{L_2(T)}.
\]
This approach provides a measure of local discretization error and can serve as the basis for the adaptive refinement strategy.

Now in order to solve the discretized Friedrichs' system \eqref{eq:DGfully} we first compute the next step $\bm U^{n+1}$ on the current grid and evaluate the error indicators for $\bm U^{n+1}$ as above. We mark an element $T$ for refinement if the error $e^{(T)}$ exceeds a predefined threshold $\epsilon$. Once the marked elements are identified, they are refined accordingly. This refinement process may lead to nonconforming meshes with hanging nodes. However, due to the local nature of the DG method, these nonconformities do not pose a problem, as continuity across element interfaces is not enforced.

After refinement, the current solution $\bm{U}^n$ is interpolated onto the new mesh. The time-stepping procedure is then repeated on the refined mesh to obtain a new $\bm U^{n+1}$, incorporating the newly adapted elements. This process is repeated until no further refinement is triggered.

In order to remove unnecessary refinements and maintain computational efficiency, a coarsening procedure is applied. The same error indicator is used to assess whether refinement is still required. Specifically, all children of a parent element are removed if the sum of their error indicators falls below $c \,\epsilon$, where $0< c < 1$ is a safety factor.

\section{Numerical simulations}\label{sec: numerical simulations}

We present numerical experiments to validate our theoretical findings and evaluate the performance of the proposed methods. We first simulate classical Landau damping in one spatial dimension, verifying the conservation properties and accuracy of the DG discretization. Additionally, we examine our modified projector splitting integrator in this setting. Next, we apply the conservative unconventional integrator to Landau damping in two spatial dimensions, demonstrating its applicability to higher-dimensional problems. Finally, we simulate the free motion of a Maxwellian distribution in two spatial dimensions governed by a linear kinetic equation with a vanishing electric field. This simplified setup enables us to study the impact of both rank and spatial adaptivity within our framework, serving as an initial step toward adaptive low-rank DG methods for kinetic equations.

The implementation is based on the finite element library \texttt{MFEM}~\cite{mfem2021}. All computations were performed on a workstation equipped with dual-socket AMD EPYC 7543 processors and 1~TB~RAM.

\subsection{Landau damping in one spatial dimension}\label{sec: Landau1d1v}

We consider classical Landau damping in one spatial and one velocity dimension over the time interval $[0,40]$. We use periodic domains $\Omega^{(x)} = [0, 4 \pi]$ and $\Omega^{(v)} = [-6, 6]$ with the initial condition
\[
f(0,x,v) = \frac{1}{\sqrt{2\pi}} \mathrm e^{-v^2/2}\, \big( 1 + \alpha \cos(k x)\big), \quad
\alpha = 10^{-2}, \, k = \frac 1 2.
\]
The same setup was investigated in~\cite{einkemmer2023}. In this case, linear analysis predicts that the electric field decays with a rate of $\gamma \approx 0.153$. For the discretization, we use a uniform grid with $n_x=32$ elements in the spatial domain and $n_v=64$ elements in the velocity domain. Quadratic finite elements ($p=2$) are employed, along with a step size $\tau = 10^{-4}$.

\begin{figure}[t]
\begin{center}
\includegraphics[width=0.89\textwidth]{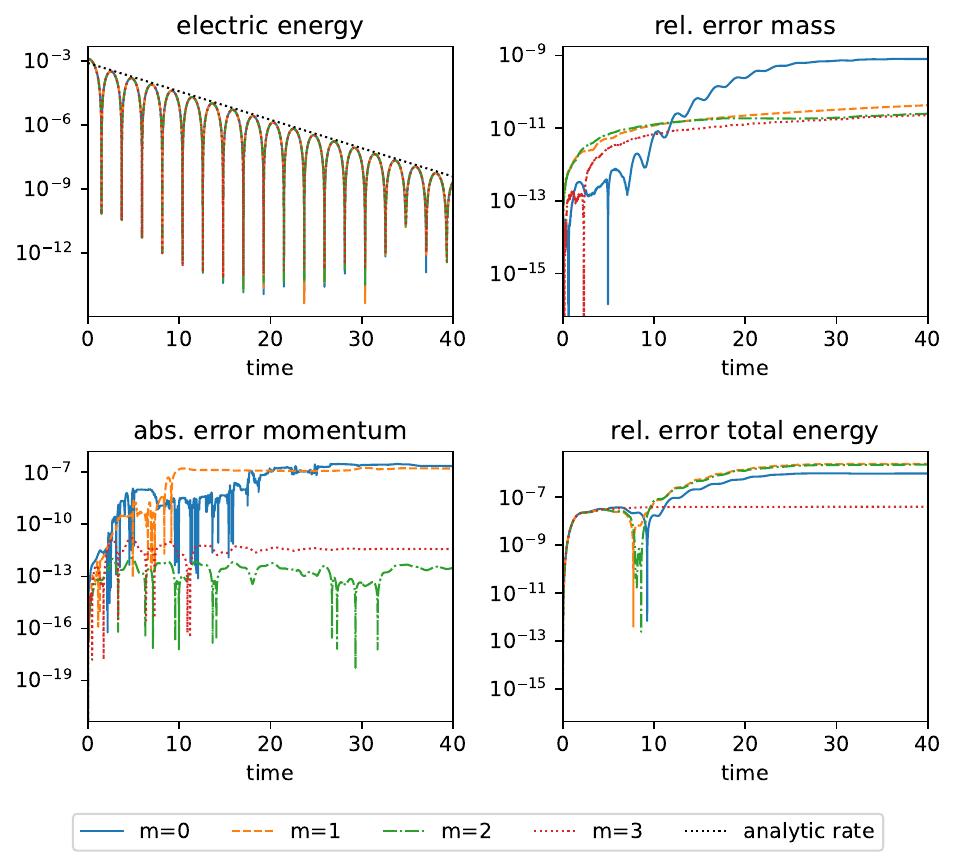}
\end{center}
\caption{Simulation results of 1+1-dimensional linear Landau damping using the unconventional integrator with a maximum rank of $10$, central flux, and different number of fixed velocity basis functions $m$; electric energy including the theoretical decay rate (upper right), relative error of the mass (upper right) and total energy (bottom right), absolute error of the momentum (bottom left) \label{fig:landau1d_bug_fixed_central}}
\end{figure}

The results for the conservative unconventional integrator with truncation to a fixed rank of $10$, shown in Figure~\ref{fig:landau1d_bug_fixed_central}, closely resemble those in~\cite[Fig.~1]{einkemmer2023}. While all configurations ($m=0,1,2,3$) exhibit the expected decay in electric energy, their behavior with respect to physical invariants differs significantly. Including the  constant function $v \mapsto 1$ for $m=1,2,3$ reduces the relative error in mass from approximately  $10^{-9}$ to $5 \cdot 10^{-11}$. The absolute error of the momentum decreases from approximately $10^{-7}$ to $10^{-11}$ when the function $v \mapsto v$ is included ($m=2,3$). Consequently, both quantities can be considered numerically conserved for $m \ge 2$. As argued in~\cite{einkemmer2023} the total energy is not expected to be conserved; however, including $v \mapsto v^2$ reduces the error by approximately two orders of magnitude for $m=3$.

\begin{figure}[t]
\begin{center}
\includegraphics[width=0.89\textwidth]{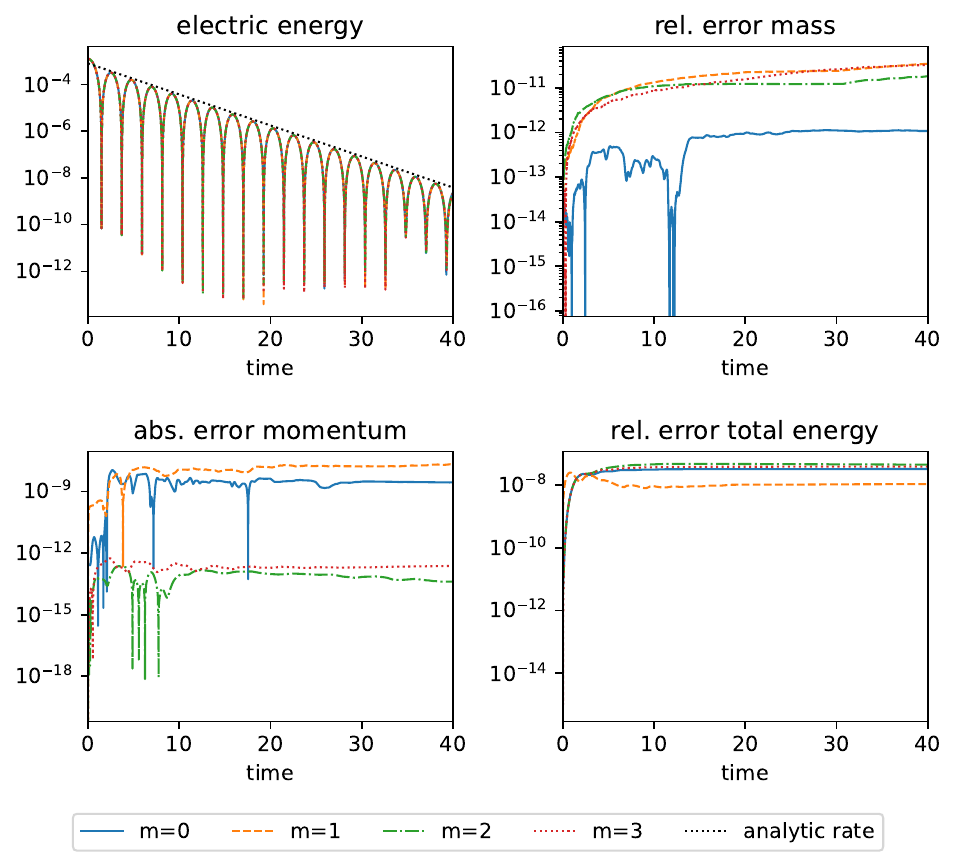}
\end{center}
\caption{Simulation results of 1+1-dimensional linear Landau damping using the modified projector splitting integrator with rank $10$, central flux, and different number of fixed velocity basis functions $m$; electric energy including the theoretical decay rate (upper right), relative error of the mass (upper right) and total energy (bottom right), absolute error of the momentum (bottom left)}
\label{fig:landau1d_projector_central}
\end{figure}

We repeat the experiment using the projector splitting integrator to assess its ability to conserve physical invariants. The results, shown in Figure~\ref{fig:landau1d_projector_central}, confirm that the integrator correctly reproduces the expected decay of electric energy according to the theoretical prediction. However, the conservation of physical invariants yields mixed results. Notably, including the constant function increases the error in mass, whereas incorporating the linear function systematically improves the accuracy of the momentum. The total energy error remains small but exhibits an inconsistent behavior across different configurations.

\begin{figure}[t]
\begin{center}
\includegraphics[width=0.89\textwidth]{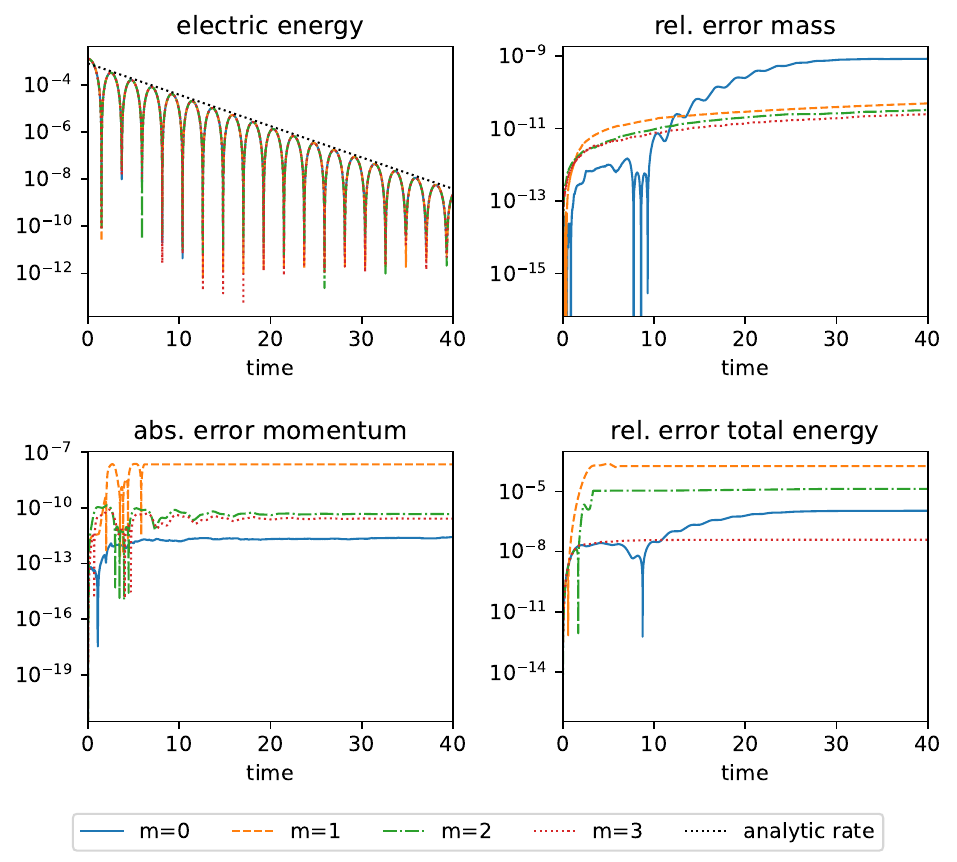}
\end{center}
\caption{Simulation results of 1+1-dimensional linear Landau damping using the rank-adaptive unconventional integrator with a truncation tolerance of $\epsilon = 10^{-7}$, central flux, and different number of fixed velocity basis functions $m$; electric energy including the theoretical decay rate (upper right), relative error of the mass (upper right) and total energy (bottom right), absolute error of the momentum (bottom left) \label{fig:landau1d_bug_adaptive_central}}
\end{figure}

Finally, we also apply the unconventional integrator in its rank-adaptive version, employing an SVD truncation with a Frobenius norm error threshold of $\epsilon = 10^{-7}$. The results, presented in Figure~\ref{fig:landau1d_bug_adaptive_central}, closely resemble those obtained with a fixed rank in Figure~\ref{fig:landau1d_bug_fixed_central}. This similarity is expected, as the maximum selected rank throughout the simulation remained below $10$. However, a notable difference is observed in the momentum error: the best results are achieved when no velocity function is fixed ($m=0$).

\subsection{Landau damping in two spatial dimensions}\label{sec: Landau 2d2v}

We now extend our simulation of Landau damping to two spatial dimensions, demonstrating that the DG-based approach generalizes well to higher dimensions. Specifically, we solve the problem on periodic domains $\Omega^{(x)} = [0, 4 \pi]^2$ and $\Omega^{(v)} = [-6, 6]^2$, using the initial condition
\[
f(0,\bm x,\bm v) = \frac{1}{2\pi} \mathrm e^{-\vert \bm v\vert^2/2}\, \big( 1 + \alpha \cos(k x_1) + \alpha \cos(k x_2)\big), \quad
\alpha = 10^{-2}, \, k = \frac 1 2.
\]

We use a uniform triangulation with $n_x=2\cdot 32^2$ and $n_v=2\cdot64^2$ elements, quadratic finite elements ($p=2$), and a step size $\tau = 10^{-4}$. Results for the rank-adaptive unconventional integrator are shown in Figure~\ref{fig:landau2d_bug_adaptive_central}. Note that in this figure $m=3$ corresponds to momentum conservation, since in the two-dimensional setup two additional velocity functions~$\bm v \mapsto v_s$ ($s=1,2$) must be included. Additionally, we present the norm of the momentum error.

\begin{figure}[t]
\begin{center}
\includegraphics[width=0.89\textwidth]{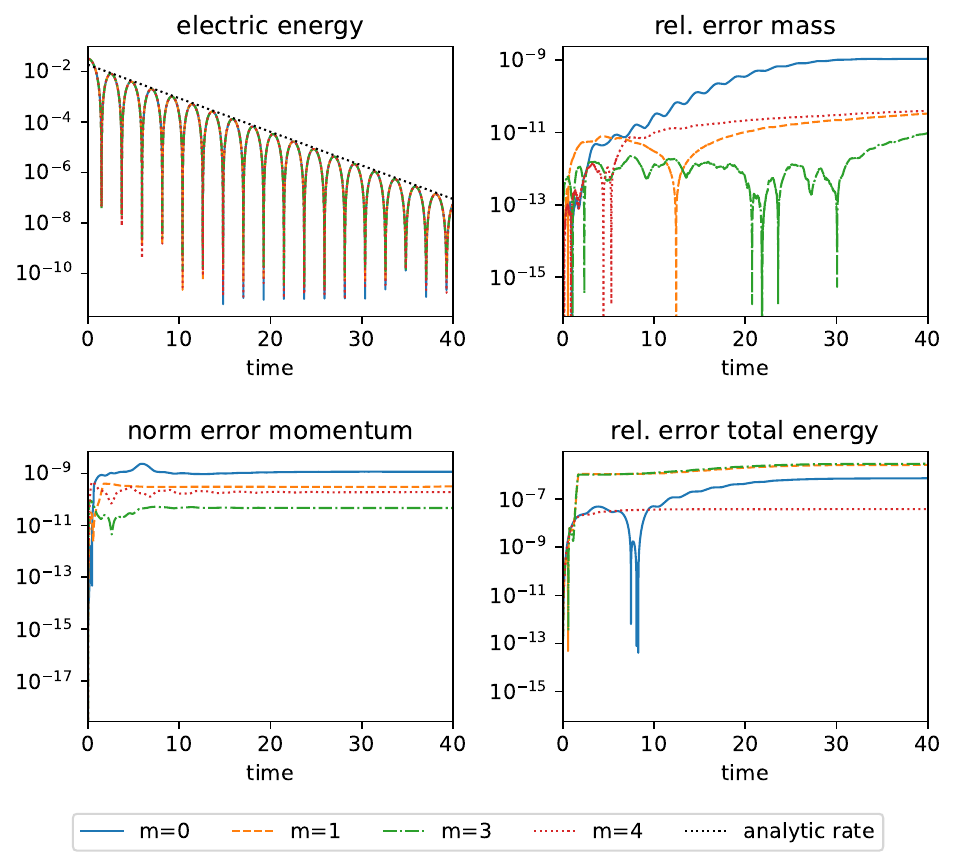}
\end{center}
\caption{Simulation results of 2+2-dimensional linear Landau damping using the rank-adaptive unconventional integrator with a truncation tolerance of $\epsilon = 10^{-7}$, central flux, and different number of fixed velocity basis functions $m$; electric energy including the theoretical decay rate (upper right), relative error of the mass (upper right) and total energy (bottom right), absolute error of the momentum (bottom left) \label{fig:landau2d_bug_adaptive_central}}
\end{figure}

The results are quite similar to the corresponding one-dimensional case in Figure~\ref{fig:landau1d_bug_adaptive_central}. The only notable difference appears in the momentum conservation. Even for $m=1$, the results are nearly as accurate as when including the two additional basis functions.

\subsection{Mesh- and rank-adaptive transport problem}\label{sec: transport adaptive DG}

As a final experiment, we conduct a first test combining spatial adaptivity with a rank-adaptive scheme. Mesh coarsening poses additional challenges regarding conservation properties, which are beyond the scope of this work. Hence, we do not take conservation into account in this experiment. Consequently, we are not fixing any velocity basis functions ($m=0$) and do not use a weight~function.

As a model problem, we consider the Vlasov--Poisson equation~\eqref{eq:transport} with the electric field $\bm E(t,\bm x)=\bm 0$ set to zero at all times on the periodic domain $\Omega = [0,4\pi]^2 \times [-6,6]^2$. The initial condition is chosen as a Maxwellian distribution with a nonzero mean velocity in the $x$-direction:
\[
f(0,\bm x,\bm v) = \frac{1}{2\pi} \mathrm e^{-\vert \bm v- \bm \mu^{(v)}\vert^2/(2\sigma_v^2)}\, \frac{1}{2\pi} \mathrm e^{-\vert \bm x - \bm \mu^{(x)}\vert^2/(2\sigma_x^2)}
\]
with
\[
\sigma_x = 1/2, \quad \bm \mu^{(x)} = [\pi,2\pi]^T, \quad \sigma_v = 1/4, \quad \bm \mu^{(v)} = [\pi,0]^T.
\]
This corresponds to the free motion of the distribution. We simulate the time interval $[0,2]$.

The initial discretization is based on a regular triangulation with $n_x=2 \cdot 16^2$ and $n_v = 2\cdot 32^2$ elements and a time step of $\tau = 5 \cdot 10^{-3}$. We estimate the error by projecting onto discontinuous linear finite elements, as described in section~\ref{sec: spatial adaptivity}. An element is refined if the error exceeds $10^{-3}$ and all children of a parent element are coarsened if the sum of their error indicators falls below $0.15 \cdot 10^{-3}$. The rank adaptive integrator is applied with a truncation tolerance of $10^{-4}$.
The time step $\tau$ itself was fixed to a small enough value to fulfill CFL conditions at all times and for all meshes. In principle the time step could be adapted to the grid or substepping could be applied, but this was not implemented.

Figure~\ref{fig:transport2d2v_mesh} displays the spatial density distribution $\rho(t,\bm x)$ at different times, along with the corresponding adaptive mesh. At $t=0$, the initially coarse, regular mesh (left plot) is refined around the center of the initial state to accurately represent the function with an error below the given tolerance. As the distribution evolves under the rank-adaptive scheme, its center shifts to the right, and it broadens in the spatial domain. The mesh adapts to the changing distribution, refining at the new location while coarsening in regions where the distribution becomes negligible (middle plot $t=1$ and right plot $t=2$). These results demonstrate that the spatial adaptivity is effective in this setting, even with a simple error indicator and refinement strategy.

We also display the number of mesh elements as well as the rank over time in Figure~\ref{fig:transport2d2v_adaptive}. The number of elements in the spatial domain increases only slightly during the simulation, corresponding to the broadening of the spatial distribution. The rank increases since the evolving distribution becomes more and more entangled in phase space.

\begin{figure}[t]
\begin{center}
\includegraphics[width=0.31\textwidth]{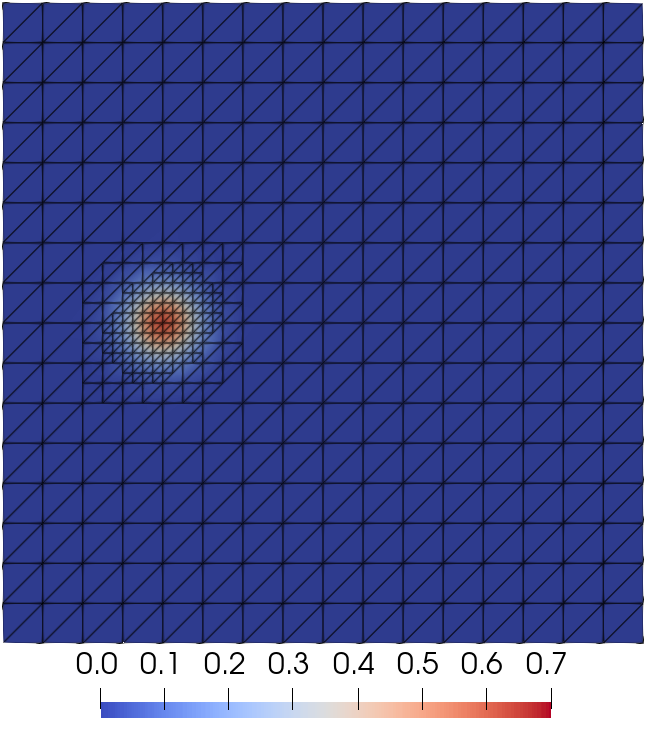} \hfill
\includegraphics[width=0.31\textwidth]{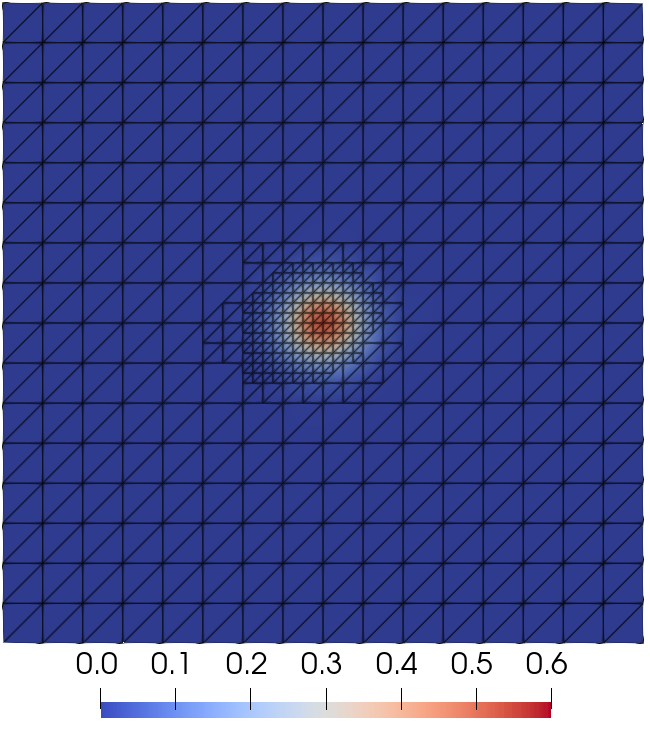} \hfill
\includegraphics[width=0.31\textwidth]{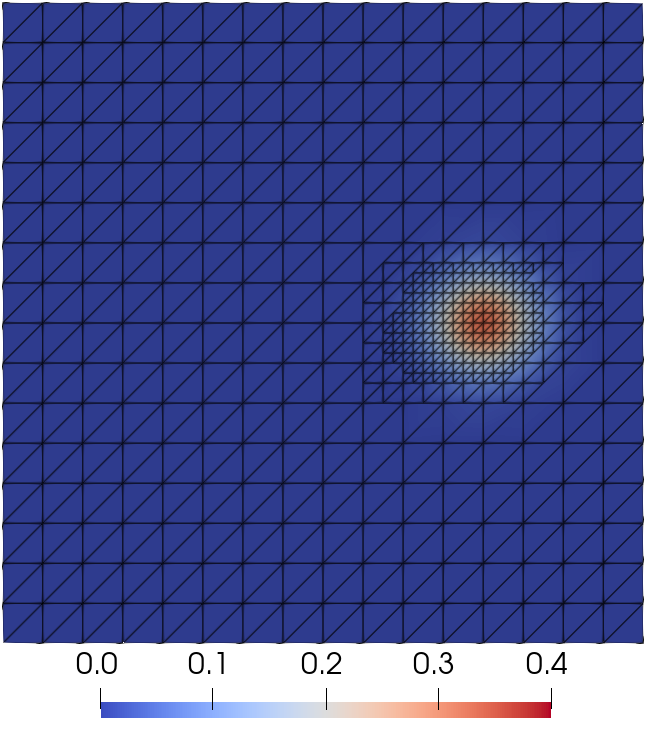}
\end{center}
\caption{Spatial density at $t=0$ (left), $t=1$ (middle), and $t=2$ (right) of the free motion of a Maxwellian distribution along with the dynamically adaptive mesh}
\label{fig:transport2d2v_mesh}
\end{figure}

\begin{figure}[t]
\begin{center}
\includegraphics[width=0.63\textwidth]{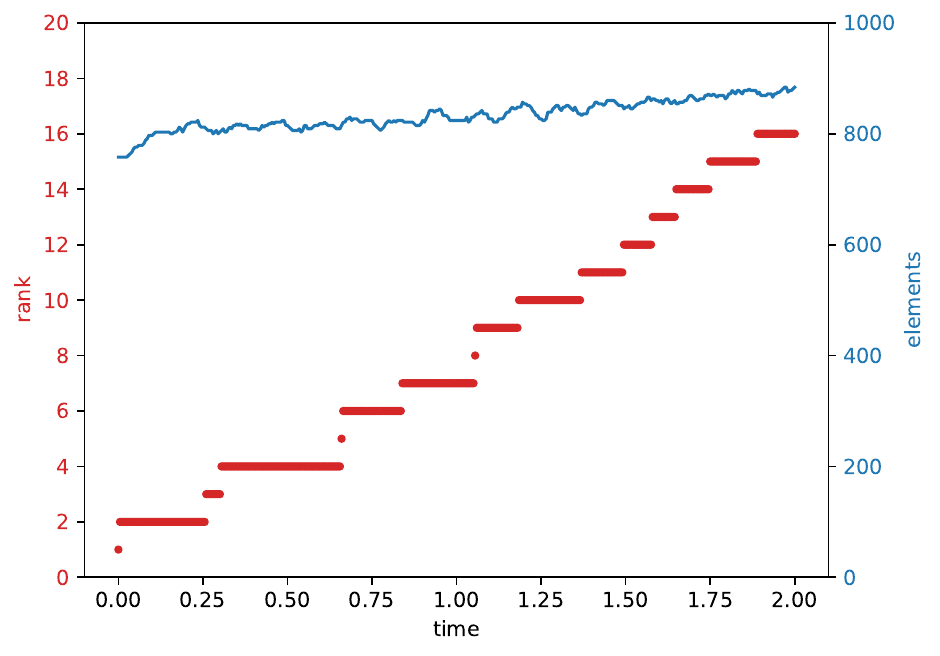}
\end{center}
\caption{Number of mesh elements in the spatial domain (blue) and DLRA rank (red) for the simulation of a free motion of a Maxwellian using a mesh- and rank-adaptive scheme}
\label{fig:transport2d2v_adaptive}
\end{figure}

\section{Conclusion}

We developed and analyzed a conservative DG discretization of the DLRA approach for the numerical solution of the Vlasov--Poisson equation. We presented the conservative integrator from~\cite{einkemmer2021,einkemmer2023,einkemmer2025}, which is based on fixing some velocity basis functions, in a modified form that avoids normal equations in the L-step and suggests a projector splitting scheme. We rigorously derived the DG formulation for the DLRA equations in terms of Friedrichs' systems, including appropriate numerical fluxes and discrete differentiation operators. Combined with the unconventional rank-adaptive integrator, we proved the conservation of mass and momentum in our numerical solver. The numerical experiments demonstrate the feasibility of our approach, confirm the conservation of physical invariants, and highlight the possibility for mesh adaptivity.

Despite these advances, several questions remain open and could be addressed in future work. First, extending the method to non-periodic domains and corresponding boundary conditions requires additional considerations, potentially based on the ideas in~\cite{uschmajew2024}. Second, more sophisticated error estimators and mesh adaptive strategies need to be designed for the DG solvers to maintain the conservation of invariant quantities. Also, higher-order time-stepping schemes, such as those suggested in~\cite{nobile2025a}, could be considered. So far, we have established conservation properties only for the unconventional integrator, but certain quantities may also be conserved under the projector splitting approach, which might merit a more detailed investigation.

\paragraph*{Acknowledgement} 

We would like to thank Michael Dumbser for helpful discussion on the DG method. We are also grateful to the anonymous referees for their valuable feedback which helped to improve the presentation and simplify some of the main results. The work of A.U.~was supported by the Deutsche Forschungsgemeinschaft (DFG, German Research Foundation) – Projektnummer 506561557.

\small

\begin{thebibliography}{10}

\bibitem{Allmann-Rahn2022}
F.~Allmann-Rahn, R.~Grauer, and K.~Kormann.
\newblock A parallel low-rank solver for the six-dimensional {V}lasov-{M}axwell
  equations.
\newblock {\em J. Comput. Phys.}, 469:Paper No. 111562, 18, 2022.

\bibitem{mfem2021}
R.~Anderson, J.~Andrej, A.~Barker, and et~al.
\newblock M{FEM}: {A} modular finite element methods library.
\newblock {\em Comput. Math. Appl.}, 81:42--74, 2021.

\bibitem{Ceruti2024}
G.~Ceruti, L.~Einkemmer, J.~Kusch, and C.~Lubich.
\newblock A robust second-order low-rank {BUG} integrator based on the midpoint
  rule.
\newblock {\em BIT}, 64(3):Paper No. 30, 19, 2024.

\bibitem{Ceruti2022}
G.~Ceruti, M.~Frank, and J.~Kusch.
\newblock Dynamical low-rank approximation for {Marshak} waves.
\newblock Technical Report~76, {Karlsruher Institut f{\"{u}}r Technologie
  (KIT)}, 2022.

\bibitem{CerutiKuschLubich2022}
G.~Ceruti, J.~Kusch, and C.~Lubich.
\newblock A rank-adaptive robust integrator for dynamical low-rank
  approximation.
\newblock {\em BIT}, 62(4):1149--1174, 2022.

\bibitem{Deriaz2018}
E.~Deriaz and S.~Peirani.
\newblock Six-dimensional adaptive simulation of the {V}lasov equations using a
  hierarchical basis.
\newblock {\em Multiscale Model. Simul.}, 16(2):583--614, 2018.

\bibitem{dolejsi2015}
V.~Dolej{\v s}{\'i} and M.~Feistauer.
\newblock {\em Discontinuous {G}alerkin method}.
\newblock Springer, Cham, 2015.

\bibitem{einkemmer2021}
L.~Einkemmer and I.~Joseph.
\newblock A mass, momentum, and energy conservative dynamical low-rank scheme
  for the {V}lasov equation.
\newblock {\em J. Comput. Phys.}, 443:Paper No. 110495, 16, 2021.

\bibitem{einkemmer2024review}
L.~Einkemmer, K.~Kormann, J.~Kusch, R.~G. McClarren, and J.-M. Qiu.
\newblock A review of low-rank methods for time-dependent kinetic simulations.
\newblock {\em arXiv:2412.05912}, 2024.

\bibitem{einkemmer2025}
L.~Einkemmer, J.~Kusch, and S.~Schotth{\"o}fer.
\newblock Construction of high-order conservative basis-update and {{Galerkin}}
  dynamical low-rank integrators.
\newblock {\em arXiv:2311.06399v3}, 2025.

\bibitem{einkemmer2018a}
L.~Einkemmer and C.~Lubich.
\newblock A low-rank projector-splitting integrator for the {V}lasov-{P}oisson
  equation.
\newblock {\em SIAM J. Sci. Comput.}, 40(5):B1330--B1360, 2018.

\bibitem{einkemmer2023}
L.~Einkemmer, A.~Ostermann, and C.~Scalone.
\newblock A robust and conservative dynamical low-rank algorithm.
\newblock {\em J. Comput. Phys.}, 484:Paper No. 112060, 20, 2023.

\bibitem{ern_2021}
A.~Ern and J.-L. Guermond.
\newblock {\em Finite elements {III}---first-order and time-dependent {PDE}s}.
\newblock Springer, Cham, 2021.

\bibitem{guo2016}
W.~Guo and Y.~Cheng.
\newblock A sparse grid discontinuous {G}alerkin method for high-dimensional
  transport equations and its application to kinetic simulations.
\newblock {\em SIAM J. Sci. Comput.}, 38(6):A3381--A3409, 2016.

\bibitem{guo2024a}
W.~Guo, J.~F. Ema, and J.-M. Qiu.
\newblock A local macroscopic conservative ({L}o{M}a{C}) low rank tensor method
  with the discontinuous {G}alerkin method for the {V}lasov dynamics.
\newblock {\em Commun. Appl. Math. Comput.}, 6(1):550--575, 2024.

\bibitem{guo2024b}
W.~Guo and J.-M. Qiu.
\newblock A conservative low rank tensor method for the {V}lasov dynamics.
\newblock {\em SIAM J. Sci. Comput.}, 46(1):A232--A263, 2024.

\bibitem{Hauck2023}
C.~D. Hauck and S.~Schnake.
\newblock A predictor-corrector strategy for adaptivity in dynamical low-rank
  approximations.
\newblock {\em SIAM J. Matrix Anal. Appl.}, 44(3), 2023.

\bibitem{hesthaven2008}
J.~S. Hesthaven and T.~Warburton.
\newblock {\em Nodal discontinuous {G}alerkin methods}.
\newblock Springer, New York, 2008.

\bibitem{Huang2023}
J.~Huang, W.~Guo, and Y.~Cheng.
\newblock Adaptive sparse grid discontinuous {G}alerkin method: review and
  software implementation.
\newblock {\em Commun. Appl. Math. Comput.}, 6(1):501--532, 2024.

\bibitem{Koch07}
O.~Koch and C.~Lubich.
\newblock Dynamical low-rank approximation.
\newblock {\em SIAM J. Matrix Anal. Appl.}, 29(2):434--454, 2007.

\bibitem{lubich2014}
C.~Lubich and I.~V. Oseledets.
\newblock A projector-splitting integrator for dynamical low-rank
  approximation.
\newblock {\em BIT}, 54(1):171--188, 2014.

\bibitem{nobile2025a}
F.~Nobile and S.~Riffaud.
\newblock Robust high-order low-rank {BUG} integrators based on explicit
  {R}unge-{K}utta methods.
\newblock {\em arXiv:2502.07040}, 2025.

\bibitem{uschmajew2024}
A.~Uschmajew and A.~Zeiser.
\newblock Dynamical low-rank approximation of the {V}lasov-{P}oisson equation
  with piecewise linear spatial boundary.
\newblock {\em BIT}, 64(2):Paper No. 19, 26, 2024.

\bibitem{Verboncoeur2005}
J.~P. Verboncoeur.
\newblock Particle simulation of plasmas: Review and advances.
\newblock {\em Plasma Phys. Control. Fusion}, 47:A231--A260, 2005.

\bibitem{Wieners2023}
C.~Wieners.
\newblock A space-time discontinuous {G}alerkin discretization for the linear
  transport equation.
\newblock {\em Comput. Math. Appl.}, 152:294--307, 2023.

\end{thebibliography}
\bibliographystyle{plain}

\end{document}